\documentclass[english,11pt]{amsart}
\usepackage{comment}

\usepackage{float}
\usepackage{mathtools}
\usepackage{amstext}
\usepackage{amsthm}
\usepackage{amsmath}
\usepackage{amssymb}
\usepackage{algorithm}
\usepackage{geometry}
\usepackage[backref=page,colorlinks,citecolor=blue,bookmarks=true]{hyperref}\hypersetup{linkcolor=blue,  citecolor=blue, urlcolor=blue}

\usepackage{tikz}
\usetikzlibrary{shapes}

\usepackage{nicefrac,xfrac}

\floatname{algorithm}{Algorithm}

\makeatletter

\numberwithin{equation}{section}
\numberwithin{figure}{section}
\theoremstyle{plain}
\newtheorem{thm}{\protect\theoremname}[section]
\DeclareMathOperator*{\Ex}{\mathbb{E}}
\DeclareMathOperator*{\Pro}{\mathbb{P}}

\newtheorem{conj}[thm]{Conjecture}

\newtheorem{cor}[thm]{Corollary}
\newtheorem*{thm*}{\protect\theoremname}
\theoremstyle{definition}
\newtheorem{problem}[thm]{\protect\problemname}
\newtheorem*{problem*}{Problem}
\theoremstyle{remark}
\newtheorem*{rem*}{\protect\remarkname}
\theoremstyle{remark}
\newtheorem{rem}[thm]{\protect\remarkname}

\theoremstyle{definition}
\newtheorem{defn}[thm]{\protect\definitionname}
\theoremstyle{plain}
\newtheorem{prop}[thm]{\protect\propositionname}
\theoremstyle{plain}
\newtheorem{fact}[thm]{\protect\factname}
\theoremstyle{definition}

\theoremstyle{plain}
\newtheorem{lem}[thm]{\protect\lemmaname}
\theoremstyle{plain}
\newtheorem{claim}[thm]{Claim}
\theoremstyle{plain}
\usepackage[backref=page,colorlinks,citecolor=blue,bookmarks=true]{hyperref}\hypersetup{linkcolor=blue,  citecolor=blue, urlcolor=blue}
\usepackage[fontsize=11pt]{scrextend}

\usepackage[nobysame,abbrev,alphabetic]{amsrefs}
\usepackage{algorithmic}

\usepackage{microtype}

\floatname{algorithm}{Algorithm}

\let\originalleft\left
\let\originalright\right
\renewcommand{\left}{\mathopen{}\mathclose\bgroup\originalleft}
\renewcommand{\right}{\aftergroup\egroup\originalright}

\makeatother

\usepackage{babel}
\makeatletter
\addto\extrasfrench{%
   \providecommand{\fg}{\ifdim\lastskip>\z@\unskip\fi~\frqq}%
}

\makeatother
\addto\captionsenglish{\renewcommand{\definitionname}{Definition}}
\addto\captionsenglish{\renewcommand{\factname}{Fact}}
\addto\captionsenglish{\renewcommand{\lemmaname}{Lemma}}
\addto\captionsenglish{\renewcommand{\problemname}{Problem}}
\addto\captionsenglish{\renewcommand{\propositionname}{Proposition}}
\addto\captionsenglish{\renewcommand{\remarkname}{Remark}}
\addto\captionsenglish{\renewcommand{\theoremname}{Theorem}}
\addto\captionsfrench{}
\addto\captionsfrench{\renewcommand{\definitionname}{Definition}}
\addto\captionsfrench{\renewcommand{\factname}{Fait}}
\addto\captionsfrench{\renewcommand{\lemmaname}{Lemme}}
\addto\captionsfrench{\renewcommand{\problemname}{Problème}}
\addto\captionsfrench{\renewcommand{\propositionname}{Proposition}}
\addto\captionsfrench{\renewcommand{\remarkname}{Remarque}}
\addto\captionsfrench{\renewcommand{\theoremname}{Théorème}}
\providecommand{\definitionname}{Definition}
\providecommand{\factname}{Fact}
\providecommand{\lemmaname}{Lemma}
\providecommand{\problemname}{Problem}
\providecommand{\propositionname}{Proposition}
\providecommand{\remarkname}{Remark}
\providecommand{\theoremname}{Theorem}
\providecommand{\examplename}{Example}

\newcommand{\eps}{\varepsilon}

\newcommand{\Sym}{{\rm Sym}}

\newcommand{\vol}{{\rm Vol}}

\newcommand{\Id}{{\rm Id}}

\newcommand{\FF}{\mathbb{F}}
\newcommand{\ZZ}{\mathbb{Z}}
\newcommand{\RR}{\mathbb{R}}
\newcommand{\QQ}{\mathbb{Q}}

\newcommand{\cB}{\mathcal{B}}

\newcommand{\cG}{\mathcal{G}}
\newcommand{\cX}{\mathcal{X}}
\newcommand{\cY}{\mathcal{Y}}
\newcommand{\cZ}{\mathcal{Z}}

\newcommand{\Img}{\textrm{Im}}
\newcommand{\Ker}{\textrm{Ker}}

\newcommand{\raE}{\overrightarrow{E}}
\newcommand{\raP}{\overrightarrow{P}}

\title[Stability of Homomorphisms, Coverings and Cocycles]{Stability of Homomorphisms, Coverings and Cocycles  {II}:\\ Examples, Applications and Open problems}

\author[M.\ Chapman]{Michael Chapman}
\address{Michael Chapman\hfill\break
	Courant Institute of Mathematical Sciences\hfill\break
	New York University,\hfill\break 251 Mercer St, New York, NY 10012, USA.}
\email{mc9578@nyu.edu}

\author[A.\ Lubotzky]{Alexander Lubotzky}
\address{Alexander Lubotzky\hfill\break
	Weizmann institute of Science\hfill\break
	Rehovot, Israel.}
\email{alex.lubotzky@mail.huji.ac.il}

\begin{document}

\maketitle
\begin{small}
    \begin{center}
     \emph{Dedicated to Shmuel Weinberger on his $60^{\rm th}$ birthday}\\
\emph{with admiration and affection}
 \end{center}
\end{small}
 
\begin{abstract}
    Coboundary expansion (with $\FF_2$ coefficients), and variations on it, have been the focus of intensive research in the last two decades. It was used to study random complexes, property testing, and above all Gromov's topological overlapping property.

    In part I of this paper, we extended the notion of coboundary expansion (and its variations) to cochains with \textbf{permutation coefficients}, equipped with the normalized Hamming distance. We showed that this gives a unified language for studying covering stability of complexes, as well as stability of group homomorphisms --- a topic that drew a lot of attention in recent years.

    In this part, we extend the theory to the permutation coefficients setting. This gives some new results, even for $\FF_2$ coefficients, opens several new directions of research, and suggests a pattern to proving the existence of non-sofic groups. 
    Along the way, we solve the dimension $2$ case of a problem of Gromov,  exhibiting a family of bounded degree coboundary expanders with $\FF_2$ coefficients.  
\end{abstract}

\section{\textbf{Introduction}}
In Part I of this paper \cite{CL_part1}, we studied three stability problems --- that of group homomorphisms, coverings of polygonal complexes, and cocycles with permutation coefficients --- and proved they are all equivalent.  A main contribution of Part I was the definition of the \emph{cocycle} (and \emph{coboundary}) \emph{Cheeger constant} of a $2$-dimensional (polygonal) complex with coefficients in a \textbf{metric} group $\Gamma$.\footnote{The commonly used definition (cf. \cites{Dinur-Meshulam,dikstein2023coboundary}) of \emph{cocycle expansion} (usually misnomed \emph{cosystolic expansion}) with coefficients in a general group agrees with our definition when one assumes the group is  equipped with the \textbf{discrete distance} on it. Namely, every non-equal pair of elements in the coefficient group is $1$ distance apart. Thus, our definition is a generalization of these more classical notions.}  Specifically, we focused on the case where $\Gamma$ is a permutation group equipped with the normalized Hamming distance. The current paper further discusses this notion.

Coboundary  expansion with $\FF_2$ coefficients was intensively studied in the last two decades \cites{linial_meshulam2006homological,gromov2010singularities,kaufman2014high,kaufman2016isoperimetric,evra2016bounded,dikstein2023coboundary}.
It was originally defined, independently, by Linial--Meshulam  \cite{linial_meshulam2006homological} and  Gromov  \cite{gromov2010singularities}. Linial--Meshulam defined it in order to study the cohomology of random $2$-dimensional simplicial complexes, while Gromov's goal was to understand  topological overlapping. 
 In a companion paper \cite{CL_stability_Random_complexes}, the Linial--Meshulam approach is taken,  to study the cocycle and coboundary expansion of random simplicial complexes with permutation coefficients. In this paper, we follow the developments of Gromov \cite{gromov2010singularities} and his followers, substituting again the $\FF_2$ coefficients with permutations.\footnote{As $\Sym(2)\cong \FF_2$, the Cheeger constants with $\FF_2$ coefficients of a given  complex are bounded from below by its Cheeger constants with permutation coefficients, which provides further motivation to study the latter.}
To that end, we first recall the highlights of the theory in the $\FF_2$ setup, and only then describe the permutation setup, following with our contributions and challenges.
\newpage 
\subsection*{Topological overlapping, Coboundary expansion, Cocycle expansion and Large cosystoles}\label{sec:intro_top_overlap_expansion}
\subsubsection*{\textbf{The $\FF_2$ coefficients setup}}
Let $\cX$ be a simplicial complex of dimension $d$,  $\overrightarrow\cX(i)$ the set of oriented $i$-cells, $C^i(\cX,\FF_2)=\{f\colon \overrightarrow\cX(i)\to \FF_2\}$ the $i$-cochains of $\cX$ with $\FF_2$ coefficients, and $\delta=\delta_i\colon C^i(\cX,\FF_2)\to C^{i+1}(\cX,\FF_2)$ the coboundary map. 
Given a probability distribution $\mu_i$ over $\overrightarrow\cX(i)$, we can define a notion of distance between $i$-cochains by $d(\alpha,\beta)=\Pro_{\sigma\sim \mu_i}[\alpha(\sigma)\neq \beta(\sigma)],$ and thus a norm $\Vert \alpha\Vert=d(\alpha,{\bf 0})$, where ${\bf 0}$ is the constant $0$ cochain. 
We assume all our simplicial complexes are given with probability distributions $\mu_i$ on their oriented $i$-cells. 
Let $B^i(\cX,\FF_2)=\Img \delta_{i-1}\subseteq C^i(\cX,\FF_2)$ be the $i$-coboundaries of $\cX$ with $\FF_2$ coefficients and $Z^i(\cX,\FF_2)=\Ker\delta_i\subseteq C^i(\cX,\FF_2)$ the $i$-cocycles.

For $0\leq i< d$, the $i$-coboundary Cheeger constant of $\cX$ with $\FF_2$ coefficients is
    \begin{equation}\label{eq:intro_informal_coboundary_expansion_F2}
        h^B_i(\cX,\FF_2)=\inf\left\{\frac{\Vert \delta\alpha\Vert}{d(\alpha,B^i(\cX,\FF_2))}\ \middle|\  \alpha\in C^i(\cX,\FF_2)\setminus B^i(\cX,\FF_2)\right\}.
    \end{equation}
The $i$-cocycle Cheeger constant of $\cX$ with $\FF_2$ coefficients is
    \begin{equation}\label{eq:intro_informal_cocycle_expansion_F2}
        h_i(\cX,\FF_2)=\inf\left\{\frac{\Vert \delta\alpha\Vert}{d(\alpha,Z^i(\cX,\FF_2))}\ \middle|\  \alpha\in C^i(\cX,\FF_2)\setminus Z^i(\cX,\FF_2)\right\}.
    \end{equation}
The $i$-cosystole of $\cX$ with $\FF_2$ coefficients is 
 \begin{equation}\label{eq:intro_informal_large_cosystols_F2}
        CoSyst_i(\cX,\FF_2)=\inf\{\Vert \alpha \Vert \mid \alpha\in Z^i(\cX,\FF_2)\setminus B^i(\cX,\FF_2) \}.
    \end{equation}

Gromov proved in \cite{gromov2010singularities} that a family $\frak{X}=\{\cX\}$ of (finite, locally sparse) $d$-dimensional simplicial complexes satisfy the topological overlapping property --- i.e., for every continuous map $f\colon \cX\to \RR^d$, there exists a point $p\in \RR^d$ which is covered via $f$ by a  constant fraction of its  $d$-cells\footnote{The constant fraction should be uniform on $\frak{X}$, namely independent of the specific $\cX$.} --- if for every $0\leq i\leq d-1$ there is a uniform lower bound on their $i$-\textbf{coboundary} Cheeger constants with $\FF_2$ coefficients. 
He used this implication to prove the following:
\begin{thm}[\cites{gromov2010singularities}, see also \cite{linial_meshulam2006homological}]\label{thm:intro_gromov_complete_and_buildings}
    The family of complete $d$-dimensional complexes  have a uniform lower bound on  their $i$-coboundary Cheeger constants with $\FF_2$ coefficients. Thus,  they satisfy the topological overlapping property.  

    Similarly, $d$-dimensional spherical buildings  have a uniform lower bound on their coboundary Cheeger constants with $\FF_2$ coefficients (see \cite{LMM_2016expansion_building_like}), and  therefore also satisfy the topological overlapping property.
\end{thm}

Having the topological overlapping property is a high dimensional version of being a family of expander graphs.  In this spirit, we use the notions of \emph{coboundary expander, cocycle expander} and \emph{toppological expander}  to indicate that a (family of) complex(es) has a bounded from zero coboundary Cheeger constant, cocycle Cheeger constant or satisfy the topological overlapping property, respectively.
In current terminology, Gromov's result says that $\FF_2$-coboundary expanders are topological expanders.

For graphs, the main difficulty is finding families of \textbf{bounded degree} expanders. Motivated by this analogy, Gromov posed the following problems:
\begin{problem}\label{prob:top_expanders}
    For $d\geq 2$, are there families of $d$-dimensional bounded degree topological expanders?
\end{problem}
\begin{problem}\label{prob:coboundary_expanders_F2}
    For $d\geq 2$, are there families of $d$-dimensional bounded degree $\FF_2$-coboundary expanders?
\end{problem}

Problem \ref{prob:top_expanders} was resolved in full --- first  for $d=2$ in \cite{kaufman2016isoperimetric}, and later for all dimensions in \cite{evra2016bounded} --- \textbf{without} resolving Problem \ref{prob:coboundary_expanders_F2}. This was done due to the following  generalization of Gromov's implication from coboundary expansion to topological overlapping:
\begin{thm}
[\cite{DKW_cosystolic_exp_implies_overlap}]\label{thm:intro_DKW}
    Fix $d\geq 2$, and $\frak{X}=\{\cX\}$ a family of (finite, locally sparse) $d$-dimensional simplicial complexes. Assume  both the cocycle Cheeger constants
        $h_i(\cX,\FF_2)$
    and cosystoles
        $CoSyst_i(\cX,\FF_2)$ of all $\cX \in \frak{X}$ in all dimensions are uniformly bounded away from zero.
        Then the family $\frak{X}$ satisfies the topological overlapping property.
\end{thm}

The approach of Kaufman--Kazhdan--Lubotzky in \cite{kaufman2016isoperimetric} and Evra--Kaufman in \cite{evra2016bounded} was to prove that the $d$-skeleton of $d+1$-dimensional \emph{Ramanujan Complexes} (see \cite{LSV2005ramanujan}) satisfy the conditions of Theorem \ref{thm:intro_DKW}. Let us explain this approach with more details (and in slightly greater generality): Let $\cB$ be the  Bruhat--Tits building of a  $p$-adic simple Lie group $G$ of rank $d+1$. Let $\Gamma$ be a cocompact lattice of $G$. On the one hand, $\cX={} _{\Gamma}\setminus^\cB$ can have a vanishing coboundary Cheeger constant (see \cite{lubotzky1987finite}), and thus a family of such does not resolve Problem \ref{prob:coboundary_expanders_F2}. On the other hand, Evra--Kaufman proved the following:
\begin{thm} [\cite{evra2016bounded}]\label{thm:intro_EK}
    Fix $d\in \mathbb{N}$ and $G$ a  $p$-adic simple Lie group of rank $d+1$. If $p$ is sufficiently large, and $\{\Gamma\}$  is the family of cocompact lattices of $G$, then the $d$-skeletons of $\frak{X}=\left\{{}_{\Gamma}\setminus ^\cB\right\}$ have cocycle Cheeger constants and cosystoles bounded away from zero (in all dimensions in a uniform manner), thus satisfying the conditions of Theorem \ref{thm:intro_DKW}. Hence, it forms a family of topological expanders.
\end{thm}

\subsubsection*{\textbf{The permutation coefficients setup and our contributions}}
By swapping  in  \eqref{eq:intro_informal_coboundary_expansion_F2}, \eqref{eq:intro_informal_cocycle_expansion_F2} and \eqref{eq:intro_informal_large_cosystols_F2} the coefficients $\FF_2$ with $\Sym=\bigsqcup\Sym(n)$ --- equipped with the normalized Hamming distance with errors (see \eqref{eq:permutation_normalized_Hamming_distance}) --- we get the notions of coboundary expansion,  cocycle expansion and  large cosystoles\footnote{Actually, the definition of $CoSyst_1(\cX,\Sym)$ requires some special care, and is not just plugging $\Sym$ instead of $\FF_2$ in the definition. See \eqref{eq:cosyst_sym_coeff} for the formal definition.} \textbf{in permutations}. Since permutation groups are non-abelian, we define these notions only in dimensions $0$ and $1$. In the first part of this paper \cite{CL_part1}, we related the Cheeger constants with permutations to well known problems in groups theory  and  combinatorial topology, i.e., pointwise group stability and rigidity of covering spaces. 
The goal of this paper is to mimic the developments made in the $\FF_2$ setup  for permutation coefficients.
In Section \ref{sec:Examples}, we follow Gromov and prove an analogue of Theorem  \ref{thm:intro_gromov_complete_and_buildings}:
\begin{thm}
    Complete complexes and spherical buildings are coboundary expanders with permutation coefficients.
\end{thm}
We further show in this section that every (finitely presented) group $\Gamma$ has a presentation with an arbitrarily small cocycle Cheeger constant in permutations. 
In Section \ref{sec:large_cosystols}, we study cosystoles with permutation coefficients. We specifically show the following: 
\begin{thm}\label{thm:intro_exp_of_coverings_implies_cosystols}
    Let $\cX$ be a finite polygonal complex.\footnote{A polygonal complex is a combinatorial version of a finite $2$-dimensional CW complex. See Section \ref{sec:prelim} for a formal definition.} 
    \begin{enumerate}
    \item  If the fundamental group $\pi_1(\cX,*)$ is residually finite and amenable, then for every $\eps>0$ there is a finite covering $\cY$ of $\cX$ with $CoSyst_1(\cY,\Sym)<\eps$.
        \item If the fundamental group $\pi_1(\cX,*)$ has property $(\tau)$\footnote{See Section \ref{sec:prop_tau_and_edge_exp_of_coverings} or \cite{lubotzky1994discrete} for a  definition.} ---  in particular, if it has Kazhdan property (T) --- then, for every finite connected covering $\cY$ of $\cX$ we have
            $CoSyst_1(\cY,\Sym)\geq \kappa,$
        where $\kappa>0$ depends only on $\cX$.  
        \item By applying Garland's method, we can (again) deduce  that for every finite connected covering $\cY$ of $\cX$ we have $CoSyst_1(\cY,\Sym)\geq \kappa$,
        where $\kappa>0$ depends only on the local spectral expansion\footnote{See Section \ref{sec:local_global_exp} or \cite{Dikstein_lecture_notes_trickling_down} for a formal definition.} of $\cX$.  
    \end{enumerate}
\end{thm}
 Note that the local spectral expansion of a simplicial complex $\cX$ which is covered by a Bruhat--Tits building $\cB$ of a $p$-adic Lie group $G$ depends only on $\cB$. Hence, the bound guaranteed by clause $(3)$ of Theorem \ref{thm:intro_exp_of_coverings_implies_cosystols} is valid uniformly for all the finite quotients of $\cB$, and not merely to the covers of a specific quotient $\cX$.  
 
 Having large cosystoles with permutation coefficients implies large cosystoles with $\FF_2$ coefficients. Clauses $(2)$ and $(3)$ of Theorem \ref{thm:intro_exp_of_coverings_implies_cosystols} are new even for $\FF_2$ coefficients, and  provide a new (and simple) method for proving that the complexes used in \cite{kaufman2016isoperimetric} and \cite{evra2016bounded} have large cosystoles (in dimension $1$). 

The above results on cosystoles of polygonal complexes have an interesting application to systoles of manifolds. 
\begin{thm}\label{thm:intro_Riemannian}
    Let $(M_0,g_0)$ be a closed Riemannian manifold of dimension $n$, and assume $\pi_1(M_0,*)$ has property $(\tau)$. Then, there exists a constant $c>0$ depending only on $(M_0,g_0)$, such that $Sys_{n-1}(M,\FF_2)\geq c\cdot \vol (M)$ for every finite sheeted covering $M$ of $M_0$.
\end{thm}
We refer the reader to Section \ref{sec:Riemannian_man} for a detailed formulation of the theorem and the notions therein. Here we just mention that while quite a lot is known about $Sys_1(M,\FF_2)$ of a manifold, much less is known on the higher degree systoles. Theorem \ref{thm:intro_Riemannian}, which applies also to high rank locally symmetric spaces, seems to be the first of its kind. 

In Section \ref{sec:two_dim_coboundary_expanders_F2},  we provide a solution to Problem \ref{prob:coboundary_expanders_F2} for the $d=2$ case (it remains open for $d\geq 3$). Namely, we prove that:
\begin{thm}
    There exists an infinite family of bounded degree $2$-dimensional coboundary expanders with $\FF_2$ coefficients.
\end{thm}
Though not using permutation coefficients, and thus deviating from the main theme of this paper, we provide here this construction  for two reasons: First, no such construction can be found in the literature. Second ---
\begin{thm}[Gohla--Thom \cite{Gohla_Thom}]
    If the family of complexes which we construct in Section \ref{sec:two_dim_coboundary_expanders_F2} are cocycle expanders \textbf{in permutations}, then there exist non-sofic gropus.\footnote{Actually, any cocycle stability rate in dimension $1$ with permutation coefficients for our construction will imply the existence of non-sofic groups. See Section \ref{sec:cocyc_exp_lattices_sofic}.}
\end{thm}

In fact, as discussed in Section \ref{sec:cocyc_exp_lattices_sofic}, proving cocycle expansion in permutations for one of many other arithmetic groups, should lead to a non-sofic group. See there for details.
Finding a non-sofic group is one of the most outstanding problems in group theory. The hope is that the methods already developed for the $\FF_2$ setup \cites{kaufman2016isoperimetric,evra2016bounded,dikstein2023coboundary} can be translated to the permutation coefficients setting, resolving  the search for non-sofic groups. In the rest of Section \ref{sec:open_problems} we discuss further open problems and suggest research directions related to this paper.  
\\

This paper is dedicated to Shmuel Weinberger, whose ability to connect disparate areas of mathematics is admired. We hope that the current paper is in his spirit.

\subsection{Acknowledgements}
Thanks to Yotam Dikstein, Shai Evra, Larry Guth, Nati Linial, Assaf Naor, Andrei Rapinchuk and Benjamin Weiss for useful discussions. 
Michael Chapman acknowledges with gratitude the Simons Society of Fellows and is supported by a grant from the Simons Foundation (N. 965535).
Alex Lubotzky is supported by the European Research Council (ERC)
under the European Union's Horizon 2020 (N. 882751), and by a research grant from the Center for New Scientists at the Weizmann Institute of Science.

\section{\textbf{Preliminaries}}\label{sec:prelim}

We briefly recall the main definitions and relevant results from part I \cite{CL_part1} needed for this paper.
A \emph{polygonal complex} $\cX$ is a $2$-dimensional CW complex whose $2$-cells  $\cX(2)=P(\cX)$ --- which we call \emph{polygons} --- are pasted along cyclically reduced paths of the \emph{underlying graph} ($1$-skeleton) $G(\cX)=(\cX(0)=V(\cX),\cX(1)=E(\cX))$. Every cyclically reduced path of length $\ell$ has $2\ell$ orientations, and we denote by $\overrightarrow \cX(2)=\raP(\cX)$ the collection of $2$-cells coming with a specific orientation of their pasted perimeter. We call $\raP(\cX)$ the \emph{oriented polygons}.  Given an oriented $i$-cell $x$, we denote by $[x]$ its un-oriented version.  
Let $\Gamma$ be a group and $d\colon \Gamma \times \Gamma \to \mathbb{R}_{\geq 0}$ a  bi-invariant metric on $\Gamma$.
  The \emph{$i$-cochains of $\cX$ with $\Gamma$  coefficients} are the anti-symmetric assignments of elements of $\Gamma$ to oriented $i$-cells. Namely,
    \[
        C^i(\cX,\Gamma)=\{\alpha\colon \overrightarrow{\cX}(i)\to \Gamma\mid \forall c\in \overrightarrow{\cX}(i)\colon \alpha(\bar c)=\alpha(c)^{-1}\}.
    \]
The \emph{coboundary maps} are defined as follows. For a $0$-cochain $\alpha\colon V(\cX)\to \Gamma$, its coboundary $\delta\alpha$ is the $1$-cochain 
\[
\forall x\xrightarrow{e} y\in \raE(\cX)\ \colon\ \ \delta\alpha(e)=\alpha(x)^{-1}\alpha(y).
\]
For a $1$-cochain $\alpha\colon \raE(\cX)\to \Gamma$, its coboundary $\delta\alpha$ is the $2$-cochain
\[
\forall \pi=e_1...e_\ell\in \raP(\cX)\ \colon\ \ \delta\alpha(\pi)=\alpha(e_1)...\alpha(e_\ell).\footnote{Every $1$-cochain can be extended to oriented paths in $\cX$. Thus, $\delta \alpha$ is just this extension of $\alpha$ evaluated only on the perimeters of polygons.}
\]
We assume from now on that any polygonal complex $\cX$ is given to us with probability distributions $\mu_i$  over its oriented $i$-cells $\overrightarrow\cX(i)$. We  assume $\mu_i$ is uniform over the orientations of a given $i$-cell, and thus can translate between distributions over the oriented and un-oriented cells freely.  The collection of measures $\mu_i$ is \emph{descending} if $\mu_i(\sigma)$ is proportional to $\sum_{\tau \supset \sigma} \mu_{i+1}(\tau)$.  For any two $i$-cochains with $\Gamma$ coefficients $\alpha$ and $\beta$,  the \emph{distance} between them is 
\begin{equation}\label{eq:distance_between_cochains}
    d(\alpha,\beta)=\Ex_{x\sim \mu_i}[d
(\alpha(x),\beta(x))].
\end{equation}
Given a subset $A\subseteq C^i(\cX,\Gamma)$ and a cochain $\alpha\in C^i(\cX,\Gamma)$, we define the distance of $\alpha$ from $A$ to be 
\begin{equation}
    d(\alpha,A)=\inf\{d(\alpha,\varphi)\mid \varphi\in A\}.
\end{equation}
The \emph{norm} of an $i$-cochain with $\Gamma$ coefficients $\alpha$ is its distance to the constant identity cochain, namely 
\begin{equation}\label{eq:norm_of_cochain}
    \lVert\alpha\rVert=\Ex_{x\sim \mu_i}[d
(\alpha(x),\Id)].
\end{equation}
An \emph{$i$-cocycle} is an $i$-cochain $\alpha$ for which $\delta\alpha$ is the constant identity function. We denote the collection of $i$-cocycles by $Z^i(\cX,\Gamma)$.
A $0$-cochain $\beta\colon V(\cX)\to \Gamma$ is said to be a \emph{$0$-coboundary} if it is constant. Namely, for every $x,y\in V(\cX)$ we have $\beta(x)=\beta(y)$. A $1$-cochain $\alpha\colon \raE(\cX)\to \Gamma$ is said to be a \emph{$1$-coboundary} if it is in the image of the coboundary operator $\delta\colon C^0(\cX,\Gamma)\to C^1(\cX,\Gamma)$. Namely, there exists a $0$-cochain $\beta\colon V(\cX)\to \Sym(n)$ such that for every $xy\in \raE(\cX)$, $\alpha(xy)=\delta\beta(xy)=\beta(x)^{-1}\beta(y)$. We denote by $B^i(\cX,\Sym)$ the collection of $i$-coboundaries of $\cX$.  Note that, without assuming further assumptions on $\Gamma$, the only indices for which we  define $Z^i(\cX,\Gamma)$ and $B^i(\cX,\Gamma)$ are $i=0\ \textrm{or}\ 1$.
 We say that the \emph{$i^{\rm th}$ cohomology of $\cX$ with $\Gamma$ coefficients vanishes} if
 every $i$-cocycle of $\cX$  is an $i$-coboundary. 
There is a natural action of $0$-cochains of $\cX$ with $\Gamma$ coefficients on the $1$-cochains:
\begin{equation}\label{eq:action_0-coch_on_1-coch}
    \forall \alpha \colon \raE(\cX)\to \Gamma ,\ \beta \colon V(\cX)\to \Gamma\ \colon\ \ \beta.\alpha(x\xrightarrow{e}y)=\beta(x)^{-1}\alpha(e)\beta(y).
\end{equation}
Note that for every polygon $\pi=x\xrightarrow{e_1}...\xrightarrow{e_\ell}x\in \raP(\cX)$, we have 
\begin{equation}\label{eq:conjugation_by_0_cochain_on_closed_path}
\beta.\alpha(\pi)=\beta(x)^{-1}\alpha(\pi)\beta(x).
\end{equation}
In particular, 
$\Vert\delta(\beta.\alpha)\Vert=\Vert \delta \alpha\Vert,$
and hence the action preserves $1$-cocycles. Furthermore, 
\[
d(\alpha,B^1(\cX,\Gamma))=\inf\{\Vert\beta.\alpha\Vert \mid \beta\in C^0(\cX,\Gamma)\}.
\]
 Let $\rho\colon \mathbb{R}_{\geq0}\to \mathbb{R}_{\geq0}$ be a \emph{rate function}, namely a decreasing function satisfying $\rho(\eps)\xrightarrow{\eps \to 0}0$.
    A polygonal complex is said to be  \emph{$\rho$-cocycle stable in the $i^{\rm th}$ dimension with $\Gamma$ coefficients} if for every $i$-cochain $\alpha$ we have
    \begin{equation} \label{eq:cocyc_stability}
           d(\alpha,Z^i(\cX,\Gamma))\leq \rho(\Vert\delta\alpha\Vert).
    \end{equation}
      The  $i^{\rm th}$ \emph{cocycle Cheeger constant} of a complex $\cX$ with $\Gamma$ coefficients is 
    \begin{equation}\label{eq:cocyc_Cheeger_constant}
    h_i(\cX,\Gamma)=\inf\left\{\frac{\Vert\delta\alpha\Vert}{d(\alpha,Z^i(\cX,\Gamma))}\ \middle\vert\ {\alpha \in C^i(\cX,\Gamma)},\ \Vert\delta\alpha\Vert\neq 0\right\}.
    \end{equation}
    Similarly, the  $i^{\rm th}$ \emph{coboundary Cheeger constant} of a complex $\cX$ with $\Gamma$ coefficients is 
\begin{equation}\label{eq:cobound_Cheeger_constant}
    h^B_i(\cX,\Gamma)=\inf\left\{\frac{\Vert\delta\alpha\Vert}{d(\alpha,B^i(\cX,\Gamma))}\ \middle\vert\ {\alpha \in C^i(\cX,\Gamma)} \setminus B^i(\cX,\Gamma)\right\}.
    \end{equation}
    When $h_i(\cX,\Sym)>0$ we call $\cX$ an \emph{$i$-cocycle expander with $\Gamma$ coefficients}, and when $h^B_i(\cX,\Sym)>0$ we call it an \emph{$i$-coboundary expander with $\Gamma$ coefficients}. Note that having a positive $i^{\rm th}$ cocycle Cheeger constant is equivalent to having a linear cocycle stability rate in the $i^{\rm th}$ dimension.
    Note also that if $h^B_i(\cX,\Gamma)>0$, then in particular the $i^{\rm th}$ cohomology of $\cX$ with $\Gamma$ coefficients vanishes. And that if the $i^{\rm th}$ cohomology of $\cX$ with $\Gamma$ coefficients vanishes, then $h^B_i(\cX,\Gamma)=h_i(\cX,\Gamma)$.
    \begin{rem} \label{rem:omitting_mu_from_notation}
        The stability  rate in the $i^{\rm th}$ dimension $\rho$ of a complex $\cX$ (and thus also the Cheeger constants) depends on the distributions $\mu_i$ and $\mu_{i+1}$.
    \end{rem}

For a positive integer $n$, let $\Sym(n)$ be the symmetric group acting on $[n]=\{1,...,n\}$. Given  permutations $\sigma \in \Sym(n)$ and $\tau\in \Sym(N)$ where $N\geq n$,  the \emph{normalized Hamming distance (with errors)} between them is
\begin{equation}\label{eq:permutation_normalized_Hamming_distance}
d_h(\sigma,\tau)=1-\frac{|\{i\in [n]\mid \sigma(i)=\tau(i)\}|}{N}.
\end{equation}
We may use the notation $\Vert \sigma\Vert$ for $d_h(\sigma,\Id)$, as was done in \eqref{eq:norm_of_cochain}.
Moreover, since the normalized Hamming distance can compare permutations of different sizes, we will study them in a collective manner. Let 
\[
C^i(\cX,\Sym)=\bigsqcup_{n=2}^{\infty} C^i(\cX,\Sym(n)),\ Z^i(\cX,\Sym)=\bigsqcup_{n=2}^{\infty} Z^i(\cX,\Sym(n)),\ 
B^i(\cX,\Sym)=\bigsqcup_{n=2}^{\infty} B^i(\cX,\Sym(n))
\]
be the \emph{$i$-cochains with permutation coefficients}, the \emph{$i$-cocycles with permutation coefficients} and the \emph{$i$-coboundaries with permutation coefficients} respectively. Thus, $d_h$ defines a metric on $C^i(\cX,\Sym)$ as in  \eqref{eq:distance_between_cochains}, and thus also a norm on it as in equation \eqref{eq:norm_of_cochain}. Furthermore, the definitions of cocycle stability and Cheeger constants extend naturally to the permutation coefficicents setup.

Let $\langle S|R\rangle$ be a group presentation with $|S|,|R|<\infty$, and equipped with probability distributions $\mu_S$ and $\mu_R$ on the generators and relations respectively. The \emph{presentation complex} $\cX_{\langle S|R\rangle}$ is the following polygonal complex: It has a single vertex $*$, and an edge $e(s)$ for every generator $s\in S$. Then, for every $r= s_1^{\eps_1}\cdot...\cdot s_\ell^{\eps_\ell}\in R$, we add a polygon (together with all its orientations) $\pi(r)=e(s_1)^{\eps_1}...e(s_\ell)^{\eps_\ell}$, where $e(s)^{-1}=\overline{e(s)}$ and $e(s)^{1}=e(s)$. The presentation $\langle S|R\rangle$ is \emph{$\rho$-homomorphism stable}  (in permutations) if $\cX_{\langle S|R\rangle}$ is $\rho$-cocycles stable in the $1^{\rm st}$ dimension with permutation coefficients, where $\mu_1=\mu_S$ and $\mu_2=\mu_R$. Later in the paper, we may refer to the \emph{Cheeger constant of a group presentation}.  In that, we mean the cocycle Cheeger constant in dimension $1$ with permutation coefficients of the associated presentation complex.

 A \emph{combinatorial map} $f\colon \cY\to \cX$ between two polygonal complexes is a function that maps $i$-cells of $\cY$ to $i$-cells of $\cX$ in an incidence preserving manner. Note that such a map extends uniquely to (oriented) paths in $\cY$. 
A \emph{covering} of $\cX$ is a combinatorial map $f\colon \cY\to \cX$ which is a topological covering (See  Chapter 1.3 in \cite{Hatcher_Alg_Top}). As long as $\cX$ is connected, the \emph{degree} of the covering is well defined and is equal to $|f^{-1}(x)|$ for any point $x\in \cX$. If the degree of the covering $\cY$ is $n\in \mathbb{N}$,  we call it an $n$-covering of $\cX$. 

    Let $\cX$ be a connected polygonal complex. Following Claim 6.1 in \cite{CL_part1},  there is a one to one correspondence between $n$-coverings of $G(\cX)$ and orbits of $1$-cochains of $\cX$ with $\Sym(n)$ coefficients under the action described in \eqref{eq:action_0-coch_on_1-coch}. We recall it now.
For every $1$-cochain $\alpha\colon \raE(\cX)\to \Sym(n)$, the corresponding covering $f\colon \cG \to G(\cX)$ is defined to be:
    \begin{align}
        &V(\cG)=V(\cX)\times [n];\notag\\
        &\raE(\cG)=\raE(\cX)\times [n];\notag\\
 \forall x\xrightarrow{e}y\in \raE(\cX),i\in [n]\ \colon \ \ &\tau(e,i)=(y,i),\ \iota(e,i)=(x,\alpha(e).i),\ \overline{(e,i)}=(\bar e,\alpha(e).i);\footnote{Note that the permutation $\alpha(e)$ tells us how the fiber over the terminal point $y$ is mapped to the fiber over the origin point $x$ and not the other way around. This is because of our choice of left actions.}\label{eq:cocycle_to_covering}\\
\forall x\in V(\cX),e\in \raE(\cX),i\in[n]\ \colon \ \ &f(x,i)=x,\quad f(e,i)=e.\notag
    \end{align}
 
On the other hand, if $f\colon \cG\to G(\cX)$ is an $n$-covering, then one can construct a $1$-cochain $\alpha\colon \raE(\cX)\to \Sym(n)$ as follows:
     For every $x\in V(\cX)$, $|f^{-1}(x)|=n$. Hence we can label the vertices of $f^{-1}(x)$ by $\{(x,i)\}_{i=1}^n$. Note that for each vertex there are $n!$ ways of choosing these labels. Now, for every $e\in \raE(\cX)$ and $e'\in f^{-1}(e)$ define $e'=(e,i)$ if $\tau(e')=(y,i)$.  Then, $\alpha(e).i$ is the second coordinate of $\iota(e,i)$. The different choices of labeling for the fibers $f^{-1}(x)$ would give rise to different $1$-cochains \textbf{in the same orbit} of the action of $0$-cochains.
     This correspondence sends $1$-cocyceles to  $n$-coverings of the complex $\cX$. Futhermore, $1$-coboundaries are in correspondence with disjoint unions of the base complex $\cX$. Lastly, this correspondence translates $\rho$-cocycle stability to a topological robustness of coverings --- If for a covering $\cY$  of the underlying graph $G(\cX)$ most polygons in $\cX$ lift to closed paths in $\cY$, then $\cY$ is close (in an appropriate metric on graphs) to  (the $1$-skeleton of) a cover of $\cX$.

\section{\textbf{Cheeger constants with permutation coefficients of complexes and groups}}\label{sec:Examples}

As was proved in Part I of this paper \cite{CL_part1}, homomorphism stability, covering stability and  cocycle stability are equivalent.  Furthermore, the first cocycle Cheeger constant \eqref{eq:cocyc_Cheeger_constant} of a polygonal complex $\cX$ is positive if and only if the cocycle stability rate of $\cX$ is linear. This leads to two natural problems:
\begin{enumerate}
    \item Which polygonal complexes have a positive cocycle Cheeger constant in the first dimension with permutation coefficients? Furthermore, can you bound the Cheeger constant of (naturally arising)  families of complexes in a uniform manner? In Sections \ref{sec:Cheeger_of_complete_complexes} and \ref{sec:Cheeger_of_spherical_buildings}  we provide elementary proofs that complete complexes and ${A}_2$ spherical buildings have cocycle Cheeger constants which are uniformly bounded away from zero. 
    \item By Fact 3.4, Theorem 1.2 and Theorem 1.3 in Part I \cite{CL_part1}, if $\pi_1(\cX_1,*)\cong \pi_1(\cX_2,*)$ for two polygonal complexes $\cX_1$ and $\cX_2$, then their stability rates differ by a multiplicative constant. In particular, $h_1(\cX_1,\Sym)>0$ if and only if $h_1(\cX_2,\Sym)>0$. Assume the presentation complex  of  $\Gamma\cong \langle S|R\rangle$ has a positive Cheeger constant. Can we bound  the Cheeger constant of all \textbf{presentation complexes} of $\Gamma$ uniformly? The answer turns out to be no. In fact, never! Every group has presentations with an asymptotically close to zero Cheeger constant, as we prove in Section \ref{sec:small_cheeger_of_pres}.
\end{enumerate}

Since homomorphism stability is a widely studied subject, there are many examples of presentations, and thus complexes, which are $\rho$-stable for some \textbf{unspecified} $\rho$.\footnote{This is in the spirit of Section 3.2 in Part I of this paper \cite{CL_part1}.} There are cases where meaningful bounds on $\rho$ were proven. 
For example, Becker--Mosheiff \cite{BeckerMosheiff} analyzed $\rho$ for presentations of infinite abelian groups, proving that $\rho(\eps)$ is  $\Omega(\eps^{\frac{1}{d}})$ and $O(\eps^{\frac{1}{2^d}})$, where $d$ is the infinite rank of the abelian group. Also, Lazarovich--Levit--Minsky \cite{levit_lazarovich2019surface} analyzed $\rho$ for surface groups, proving that  $\rho(\eps)$ is $\Theta(-\eps\ln \eps)$. 
Note that the above results imply that the cocycle Cheeger constant of any triangulation of a compact surface which is not a sphere is $0$. This demonstrates the discrepancy between being a cocycle expander, and being stable for some rate $\rho$ which is not necessarily linear.
Becker--Chapman \cite{BC22} proved that the multiplication table presentation of a finite group yields a presentation complex with cocycle Cheeger constant larger than $\frac{1}{3000}$, regardless of the chosen group. And, on the other hand, there are examples of group presentations that are not $\rho$-homomorphism stable for any rate $\rho$ (cf. \cite{GlebskyRivera,BLT,Ioana}).

\subsection{The complete complex is a coboundary expander with permutation coefficients}\label{sec:Cheeger_of_complete_complexes}

In this section and in Section \ref{sec:Cheeger_of_spherical_buildings}, the distributions we use on the cells of the complexes are the uniform ones. The proofs we provide work the same if we choose the uniform distribution on polygons and the induced descending measure  on the edges (see Remark \ref{rem:dist_on_edges_does_not_matter}).
We begin with a technical Lemma that appeared in Section 6.1 of the first part of this paper \cite{CL_part1}, as part of the proof of Theorem 1.1 there.

\begin{lem}\label{lem:identity_on_spanning_tree}
    Let $\cX$ be a connected polygonal complex, let $T$ be a spanning tree in $G(\cX)$ and let $*\in V(\cX)$. Then, for every $1$-cochain $\alpha\colon \raE(\cX)\to \Sym(n)$, there is a $0$-cochain $\beta\colon V(\cX)\to \Sym(n)$ such that $\beta(*)=\Id$ and $\beta.\alpha(e)=\Id$ for every $e\in T$.
\end{lem}

\begin{proof}[Proof sketch]
    Recall that $\alpha$ extends to paths in $G(\cX)$. Choose $\beta(x)=\alpha(\pi_{x\to *})$, where $\pi_{x\to *}$ is the unique non-backtracking path in $T$ from $x$ to $*$. It is straightforward to check that $\beta$ satisfies the desired conditions.
\end{proof}

\begin{prop}\label{prop:complete_complex_stable}
    Let $\cX$ be the complete $2$-dimensional simplicial complex on $d+1$ vertices, namely the $2$-skeleton of the $d$-dimensional simplex, for $d\geq 2$. Then, $\cX$ is $\rho$-cocycle stable with rate $\rho(\eps)=\frac{d-1}{d+1}\eps$. Namely, $h_1(\cX,\Sym)\geq\frac{d+1}{d-1}>1$.
\end{prop}
\begin{proof}
    Recall that the vertices of the $d$-dimesnional simplex are $V(\cX)=\{1,...,d+1\}$. Further, recall that $\raE(\cX)=\{xy\mid x\neq y\in V(\cX)\}$ and $\raP(\cX)=\{xyz\mid x\neq y\neq z\in V(\cX)\}$.\footnote{Though slightly confusing, in the notation $x\neq y\neq z$ we mean that all three elements are distinct from one another.} Let $\alpha\colon \raE(\cX)\to \Sym(n)$ be a $1$-cochain.  For every fixed $x\in V(\cX)$, we can choose a spanning tree $T$ which is the star centered at $x$. By Lemma \ref{lem:identity_on_spanning_tree}, we can conjugate $\alpha$ by a $0$-cochain $\beta_x$ to be trivial on $T$.
    Since $\delta\alpha(\Delta)$ is  conjugate to $\delta\beta_x.\alpha(\Delta)$ for every triangle $\Delta$, we have
    \[
        \begin{split}
\Vert\delta\alpha\Vert=\Ex_{x\neq y\neq z\in V(\cX)}[d_h(\delta\alpha(xyz),\Id)]=
    \Ex_{x\neq y\neq z\in V(\cX)}[d_h(\delta\beta_x.\alpha(xyz),\Id)]=(\diamondsuit).
\end{split}
   \]
   Recall that $\delta\beta_x.\alpha(xyz)=\beta_x.\alpha(xy)\beta_x.\alpha(yz)\beta_x.\alpha(zx)$. Now, by our choice, $\beta_x.\alpha$ is the identity permutation on every edge with endpoint $x$. Therefore $\delta\beta_x.\alpha(xyz)=\beta_x.\alpha(yz)$. Moreover, if we would have allowed $x=y$ or $x=z$, then $\delta\beta_x.\alpha(xyz)=\beta_x.\alpha(yz)=\Id$ in these cases. Thus, by letting the expectation go over ordered triplets $xyz$ assuming only that $y\neq z$, we have
   \[
    (\diamondsuit)=\frac{d+1}{d-1}\Ex_{x, y\neq z\in V(\cX)}[d_h(\beta_x.\alpha(yz),\Id)]
    =\frac{d+1}{d-1}\Ex_{x\in V(\cX)}[\Vert\beta_x.\alpha\Vert].
    \] 
    All in all, 
    \begin{equation}\label{eq:local_defect_analysis_complete_complex}
        \begin{split}
\Vert\delta\alpha\Vert=\frac{d+1}{d-1}\Ex_{x\in V(\cX)}[\Vert\beta_x.\alpha\Vert].
\end{split}
    \end{equation}
    
Further, the $1^{\rm st}$ cohomology of $\cX$ vanishes, namely  $B^1(\cX,\Sym)=Z^1(\cX,\Sym)$  --- this is because the fundamental group of the complete complex is trivial, and by Corollary \ref{cor:first_cohomology_vanishes_equivalence}\footnote{This result also appeared as Fact 5.11 in \cite{CL_part1}.} we deduce the vanishing of the first cohomology.  Since  $d_h(\alpha,B^1(\cX,\Sym))= \inf\{\Vert \beta.\alpha\Vert\mid \beta\colon V(\cX)\to \Sym(n) \}$ for every complex $\cX$,
\begin{equation}\label{eq:global_defect_analysis_complete_complex}
    \begin{split}
 d_h(\alpha, Z^1(\cX,\Sym))
&= \min\{\Vert \beta.\alpha\Vert\mid \beta\colon V(\cX)\to \Sym(n) \}\\
&\leq \Ex_{x\in V(\cX)}[\Vert \beta_x.\alpha\Vert].
\end{split}
\end{equation}
Combining \eqref{eq:local_defect_analysis_complete_complex} and \eqref{eq:global_defect_analysis_complete_complex}, we get the desired rate.
\end{proof}

\begin{rem}
    Since the first cohomology of $\cX$ vanishes, we also deduce that $h_1^B(\cX,\Sym)\geq \frac{d+1}{d-1}$ in this case. As $h^B_1(\cX,\FF_2)\geq h_1^B(\cX,\Sym)$, we deduce the same for $\FF_2$ coefficients. Compare with  Proposition 2.1 of \cite{linial_meshulam2006homological}, where a lower bound of $\nicefrac{1}{40}$ on $h^B_1(\cX,\FF_2)$ is given.
\end{rem}

\subsection{Spherical buildings of type $A_2$ are coboundary expanders with permutation coefficients}\label{sec:Cheeger_of_spherical_buildings}

\begin{defn}[$A_2$ spherical building]
    Let $\FF_q$ be the finite field with $q$ elements. The vertices of the $A_2$ spherical building $\mathcal{B}=\mathcal{B}(q)$ are the non-trivial linear subspaces of $\FF_q^4$. Two vertices of $\mathcal{B}$ are adjacent  if one is contained in the other (as sets). For triangles, we take the clique complex of the resulting graph --- namely, add a triangle whenever we see three vertices that are all connected to one another. 
\end{defn}

\begin{defn}
    We say that a closed path $\pi$ in $\cX$ can be \emph{filled} by polygons $\pi_1,...,\pi_k\in P(\cX)$ if one can draw a Van Kampen diagram whose perimeter is $\pi$ and its cells are the polygons. Namely, $\pi$ is a product of conjugates of $\pi_1,...,\pi_k$ in the fundamental group of $\cX$. See \cite{babson_hoffman_kahle_2011fundamental} for more about fillings.
\end{defn}

\begin{claim}\label{claim:filling_inequality}
    Let $\alpha\colon \raE(\cX)\to \Sym(n)$ be a $1$-cochain. Recall the notation $\Vert\sigma\Vert=d_h(\sigma,\Id)$ for a permutation $\sigma$. If $\pi$ is a closed path in $\cX$ which is filled by $\pi_1,...,\pi_k$, then $\Vert \alpha(\pi)\Vert \leq \sum_{i=1}^k \Vert \alpha(\pi_i)\Vert$.
\end{claim}

\begin{proof}
The fact $\pi$ can be filled by $\pi_1,...,\pi_k$ means that it can be written in a backtracking way as $$\pi=\sigma_1\pi_1\bar \sigma_1...\sigma_k\pi_k\bar \sigma_k,$$ where $\sigma_i$ are some paths in $G(\cX)$. Hence
\[
\begin{split}
  \Vert\alpha(\pi)\Vert&=\left\Vert \prod_{i=1}^k \alpha(\sigma_i\pi_i\bar \sigma_i) \right\Vert\\
  &\leq \sum_{i=1}^k  \underbrace{\Vert\alpha(\sigma_i\pi_i\bar \sigma_i)\Vert}_{\Vert\alpha(\sigma_i)\alpha(\pi_i)\alpha(\bar \sigma_i)\Vert} \\
  &=\sum_{i=1}^k \left\Vert \alpha(\pi_i) \right\Vert,
\end{split}
\]
where the second line is by the triangle inequality of the normalized Hamming metric, and the third line is by the fact $\alpha(\sigma_i)^{-1}=\alpha(\bar \sigma_i)$ and the conjugate invariance of the normalized Hamming metric.
\end{proof}

We bring the following fact without a proof.
\begin{fact} \label{fact:properties_of_building}
    Note that 
\begin{itemize}
    \item The $1$-skeleton of $\cB$ has diameter $3$.
    \item The automorphism group of $\cB$ acts transitively on the un-oriented triangles of $\cB$.
    \item Every closed path $\pi$ of length $\leq 7$ in the $1$-skeleton $G(\cB)$   can be filled by at most $25$ triangles.\footnote{This bound is not intended to be optimal, but can be deduced in a straightforward manner by a back of the envelope calculation.}
\end{itemize}
\end{fact}

\begin{prop}\label{prop:building_stable}
    The complex $\cB$ is $\rho$-cocycle stable for $\rho(\eps)=25\eps$. Namely, $h_1(\cB,\Sym)\geq \frac{1}{25}$, regardless of $q$.
\end{prop}

\begin{proof}
    Let $\alpha\colon \raE(\cB)\to \Sym(n)$ be  a  $1$-cochain. Fix a vertex $x\in \cB$ and a spanning tree $T$ of radius $3$ from $x$ (there exists such a tree because of the diameter condition in Fact \ref{fact:properties_of_building}). 
    For every $\eta\in \textrm{Aut}(\cB)$ we get a new base point $\eta(x)$ and a new spanning tree $\eta(T)$. By Lemma \ref{lem:identity_on_spanning_tree}, we can conjugate $\alpha$ by a $0$-cochain $\beta_\eta$ so that it is the identity on $\eta(T)$. As for the complete complex, the cohomology of $\cB$ vanishes, hence
    \[
    \begin{split}
         d_h(\alpha,Z^1(\cB,\Sym)))&=\min\{\Vert \beta.\alpha\Vert\mid \beta\colon V(\cX)\to \Sym(n) \}\\
         &\leq \min\{\Vert \beta_\eta.\alpha \Vert\mid \eta\in \textrm{Aut}(\cB)\}\\
         &\leq \Ex_{\eta\in\textrm{Aut}(\cB)}[\Vert \beta_\eta.\alpha \Vert] .
    \end{split}
    \] 
     Since every vertex is at most $3$ steps away from $x$ via the spanning tree $T$, for every edge $e\in \raE(\cB)$ there is a closed path $\pi_e$ of length $\leq 7$ originating from $x$ which traverses $e$ as the only edge (potentially) not in $T$. By Fact \ref{fact:properties_of_building}, there are triangles $\{\Delta^e_i\}_{i=1}^{k_{e}}\subseteq \raP(\cB)$, where $k_{e}\leq 25$, which fill up $\pi_e$. 
Now, the action of $\textrm{Aut}(\cB)$ on  the data $e, \pi_e$ and $\Delta_i^e$ preserves its properties. Namely, the path $\eta(\pi_e)$ traverses $\eta(e)$ such that it is the only edge outside of $\eta(T)$. Moreover, $\eta(\pi_e)$ is filled by $\{\eta(\Delta^e_i)\}_{i=1}^{k_{e}}$. Thus,
\[
\begin{split}
         \Vert\beta_\eta.\alpha(\eta(e))\Vert&=\Vert\beta_\eta.\alpha(\eta(\pi_e))\Vert\\
         &\leq \sum_{i=1}^{k_{e}} \Vert\beta_\eta.\alpha(\eta(\Delta^{e}_i))\Vert\\
         &=\sum_{i=1}^{k_{e}} \Vert\alpha(\eta(\Delta^{e}_i))\Vert,
     \end{split}
\]
where the second line is by Claim \ref{claim:filling_inequality}, and the third is by \eqref{eq:conjugation_by_0_cochain_on_closed_path}.
Combining our observations, we get
\[
\begin{split}
          d_h(\alpha,Z^1(\cB,\Sym)))(\alpha)&\leq \Ex_{\eta\in\textrm{Aut}(\cB)}[\Vert \beta_\eta.\alpha \Vert]\\
          &=\Ex_{\eta\in\textrm{Aut}(\cB)}\Ex_{e\in \raE(\cB)}[\Vert\beta_\eta.\alpha(\eta(e))\Vert]\\
          &\leq \Ex_{\eta\in\textrm{Aut}(\cB)}\Ex_{e\in \raE(\cB)}\left[\sum_{i=1}^{k_{e}} \Vert\delta\alpha(\eta(\Delta^{e}_i))\Vert\right]\\
          &= \Ex_{e\in \raE(\cB)}\left[ \sum_{i=1}^{k_{e}}\left(\Ex_{\eta\in\textrm{Aut}(\cB)}[\Vert\delta\alpha(\eta(\Delta^{e}_i))\Vert]\right)\right]\\
          &=(*).
     \end{split}
\]
Recall that by Fact \ref{fact:properties_of_building}, $\textrm{Aut}(\cB)$ acts transitively on the triangles of $\cB$. In particular, for a fixed triangle $\Delta_i^e$, we have
\[
\Ex_{\eta\in\textrm{Aut}(\cB)}[\Vert\delta\alpha(\eta(\Delta^{e}_i))\Vert]=\Ex_{[\Delta]\in P(\cB)}[\Vert\delta\alpha(\Delta)\Vert]=\Vert \delta \alpha\Vert.
\]
 Therefore,
\[
\begin{split}
  (*)&=\Ex_{e\in \raE(\cB)}\left[k_e\cdot \Vert \delta \alpha\Vert \right] \\
  &\leq 25\Vert \delta \alpha \Vert,
\end{split}
\]
which finishes the proof. 
\end{proof}

\begin{rem}\label{rem:dist_on_edges_does_not_matter}
    Note that in the above proof we did not use the distribution $\mu_1$ on the edges at all. We just used the fact that the distribution $\mu_2$ is uniform.
\end{rem}
\begin{rem}
    We focused on type $A_2$ spherical buildings. The approach we used is usually referred  to as the \emph{cones method} (cf. \cite{LMM_2016expansion_building_like}). This approach  can be further applied to all types of spherical buildings $\mathcal{B}$  of degree $d$ (e.g., of type $A_d$), but with the stability rate degrading as  $d$ increases. The recent  combinatorial methods of Dikstein--Dinur \cite{dikstein2023coboundary}, and specifically  \emph{color restriction}, allows to prove that the stability rate of such spherical buildings is actually independent of $d$.\footnote{Dikstein--Dinur \cite{dikstein2023coboundary} use the framework of Dinur--Meshulam \cite{Dinur-Meshulam} (See Section 4.2 in Part I \cite{CL_part1}). But, the color restriction technique can be applied in our setup, i.e., with the normalized Hamming metric replacing the discrete distance.}  

    Though the proofs that spherical buildings are stable may require intricate arguments, such as the cones method and color restriction, the fact that they are stable with good rates is quite evident. Note that spherical buildings are only ``a few steps'' from being complete complexes. This is essentially the content of Fact \ref{fact:properties_of_building}. Namely, they have a bounded diameter, and every triangle of the complete complex can be filled with a bounded amount of triangles from the building. This means that these complexes are almost as stable as the complete complex, if we use the $L^\infty$-analogues (recall the discussion at the end of Section 4.2 in Part I \cite{CL_part1}). Since we do not use the $L^\infty$-analogues but rather $L^1$-versions, i.e., we sample a polygon uniformly at random in our testers (compare Algorithms 5 and 3 in Part I \cite{CL_part1}), proving stability requires more steps.
\end{rem}

\subsection{Arbitrarily small Cheeger constant of group presentations}\label{sec:small_cheeger_of_pres}

\begin{claim}\label{claim:cheeger_arbitrarily_small}
    For every $\eps>0$ and finitely presented group $\Gamma$, there exists a finite presentation $\langle S|R\rangle \cong\Gamma$ such that 
    \[
    h_1(\cX_{\langle S|R\rangle},\Sym)\leq \eps,
    \]
    where $\cX_{\langle S|R\rangle}$ is the presentation complex of $\langle S|R\rangle$.
\end{claim}
To prove Claim \ref{claim:cheeger_arbitrarily_small} we need the following corollary of Remark  6.7 in Part I \cite{CL_part1}.
\begin{claim}\label{claim:trivial_gp_cheeger_arbitrarily_small}
    For every $\eps>0$, there is a presentation $\langle S|R\rangle$ of the \textbf{trivial group} such that 
    \[
    h_1(\cX_{\langle S|R\rangle},\Sym)\leq \eps.
    \]
\end{claim}
\begin{proof}[Proof sketch]
    If we take the complete complex on $d$-vertices and contract the spanning tree $1\to 2\to...\to d$, we are left with a polygonal complex $\cX$ with a single vertex. The resulting complex defines a presentation of its fundamental group, which is isomorphic to the trivial group, by taking $E(\cX)$ to be the generating set and $P(\cX)$ to be the relations. By Remark  6.7 in \cite{CL_part1}, this presentation has $h_1(\cX,\Sym)\leq \frac{12}{d(d-1)}$. Hence, by  choosing a large enough $d$, we are done.
\end{proof}

\begin{proof}[Proof of Claim \ref{claim:cheeger_arbitrarily_small}]
    Let $\Gamma\cong \langle S'|R'\rangle$ be a finitely presented group. By Claim \ref{claim:trivial_gp_cheeger_arbitrarily_small}, for every $\eps'>0$ there is a presentation $\langle S''|R''\rangle \cong {\bf 1}$ with $h_1(\cX_{\langle S''|R''\rangle},\Sym)\leq \eps'$. We can choose the sets $S',S''$ and $R',R''$ to be disjoint.  Then, on the one hand,
    \[
\langle S|R\rangle=\langle S',S''|R',R''\rangle \cong \langle S'|R'\rangle *\langle S''|R''\rangle \cong \Gamma * {\bf 1}\cong \Gamma.
    \]
    On the other hand, every map $f\colon S''\to \Sym(n)$ can be extended to $\hat f\colon S'\cup S''\to \Sym(n)$ which is 
 trivial on $S'$. Now, if we choose $\mu_2$ and $\mu_1$ to be uniform on the relations and generators, then $\|\delta\hat f\|=\frac{|R''|}{|R''|+|R'|}\|\delta f\|$ and similarly $d_h(\hat f,Z^1(\cX_{\langle S|R\rangle},\Sym)))=\frac{|S''|}{|S''|+|S'|}d_h( f,Z^1(\cX_{\langle S''|R''\rangle},\Sym)))$. Thus,
 \[\begin{split}
     \frac{\|\delta \hat f\|}{d_h(\hat f,Z^1(\cX_{\langle S|R\rangle},\Sym))}&=\frac{|R''|(|S''|+|S'|)}{|S''|(|R''|+|R'|)}\frac{\|\delta f\|}{d_h( f,Z^1(\cX_{\langle S''|R''\rangle},\Sym))}\\
     &\leq (1+|S'|)\cdot \frac{\|\delta f\|}{d_h( f,Z^1(\cX_{\langle S''|R''\rangle},\Sym))}\\
     &\leq (1+|S'|)\eps'.
 \end{split}
 \]
    By choosing  $\eps'=\frac{\eps}{1+|S'|}$, we deduce the claim.
\end{proof}

\section{\textbf{High dimensional expansion and Large cosystoles}}\label{sec:large_cosystols}

\subsection{High dimensional expansion} \label{sec:HDX}
For every non-trivial group $\Gamma$, 
     the $0^{\rm th}$ cohomology of $\cX$ with 
 $\Gamma$ coefficients vanishes if and only if $G(\cX)$ is connected. What does it mean for  the $1^{\rm st}$ cohomology to vanish? For permutation coefficients there is a nice characterization. Since $1$-cocycles are in correspondence with coverings of the base complex, and $1$-coboundaries are disjoint unions of the base complex, having a non-coboundry $1$-cocycle is the same as having a non-trivial finite covering of the base complex. 
 Since there is a correspondence between coverings of $\cX$ and subgroups of the fundamental group $\pi_1(\cX,*)$, having no non-trivial finite coverings is equivalent to $\pi_1(\cX,*)$ having no proper finite-index subgroups. 
In summary:
\begin{cor}\label{cor:first_cohomology_vanishes_equivalence}
    Given a polygonal complex $\cX$, the following three conditions are equivalent:
    \begin{enumerate}
        \item The first cohomology of $\cX$ with permutation coefficients vanishes.
         \item The fundamental group $\pi_1(\cX,*)$  has no finite index subgroups. 
        \item The profinite completion of its fundamental group is trivial, nemaly $\widehat{\pi_1(\cX,*)}=1$.
    \end{enumerate}
\end{cor}

As oppose to the $0^{\rm th}$ dimensional case, where vanishing of the $0^{\rm th}$ cohomology, namely connectedness, is the most basic property you would expect from an object which presumes to \emph{expand}, this is not the case in  higher dimensions. Out of the various suggested high dimensional expansion properties (cf. \cite{lubotzky2018high}), coboundary expansion is considered a very strong notion. These days, there are many motivations to study coboundary expansion (cf. \cite{kaufman2014high}), but the original motivation to study it came --- as mentioned in our introduction ---  from  topological overlapping. Specifically, from the fact that coboundary expansion implies topological overlapping. Dotterrer--Kaufman--Wagner \cite{DKW_cosystolic_exp_implies_overlap} proved  that 
 a weakened version of coboundary expansion suffices to deduce topological overlapping (Theorem \ref{thm:intro_DKW}): The simplicial complex needs to  ---
 \begin{enumerate}
     \item Be a \textbf{cocycle} expander with $\FF_2$ coefficients.
     \item Have large cosystoles with $\FF_2$ coefficients.
 \end{enumerate}
We postpone addressing condition $(1)$ with permutation coefficients for the examples used by \cites{kaufman2016isoperimetric,evra2016bounded} to Section \ref{sec:cocyc_exp_lattices_sofic}.
In this section, we concentrate on condition $(2)$ with permutation coefficients.\footnote{In \cite{DKW_cosystolic_exp_implies_overlap} they called the above combination of properties \emph{cosystolic} expansion. We will not use this term. Note also that in \cites{Dinur-Meshulam,dikstein2023coboundary}  cosystolic expansion is what we call cocycle expansion, namely they do not require the large cosystoles condition.} The naive generalization of the large cosystoles condition to our setup, by asking every non-coboundary cocycle to have a large norm, cannot work: Let $\alpha\colon \raE(\cX)\to \Sym(n)$ be a non-coboundary cocycle. Note that we can define an embedding $\omega\colon \Sym(n)\to\Sym(N)$ for any $N\geq n$, such that $$\forall\sigma\in \Sym(n),\ i>n\ \ \colon \ \ \ \omega(\sigma).i=i.$$
Composing $\alpha$ with $\omega$, we get a new non-coboundary cocycle with norm which is $\frac{n}{N}$ smaller than $||\alpha||$. Since this works for any $N$, we cannot assume a unifrom lower bound on the norm of non-coboundary cocycles. Note that the cocycle $\omega\circ\alpha$ correspndes to a covering which is the disjoint union of $N-n$ copies of the underlying $\cX$ together with the covering associated with $\alpha$. In particular, it is disconnected. Nevertheless, the notion of large cosystoles can still be saved, by lower bounding the norms of \textbf{connected} non-coboundary cocycles. 
\begin{defn}
    A $1$-cochain $\alpha\colon \raE(\cX) \to \Sym(n)$ is \emph{connected} if it encodes a connected covering of the underlying graph of the base complex --- namely, for every two indices $i,j\in [n]$ and every two vertices $x,y\in V(\cX)$, there is a path $\pi$ in $\cX$ whose endpoints are $x$ and $y$ while $\alpha(\pi).i=j$.
\end{defn} 
\begin{defn}
    A complex $\cX$ has \emph{$\theta$-large cosystoles with permutation coefficients} if every non-coboundary connected cocycle $\alpha\colon \raE(\cX)\to \Sym(n)$ satisfies $\Vert\alpha\Vert\geq \theta$. I.e., $CoSyst_1(\cX,\Sym)\geq \theta$, where
    \begin{equation}\label{eq:cosyst_sym_coeff}
        CoSyst_1(\cX,\Sym)=\inf\{\Vert \alpha \Vert \mid \alpha\ \textrm{is\ \textbf{connected}},\ \alpha\in Z^1(\cX,\Sym)\setminus B^1(\cX,\Sym) \}.
    \end{equation}
\end{defn}

    Note that $CoSyst_1(\cX,\Sym)\geq \frac{1}{|\overrightarrow{E}(\cX)|}$. Given $\cY_i$  a family of growing coverings of $\cX$, the quantity $CoSyst_1(\cY_i,\Sym)$ may approach zero. An example to that is the coverings of a cycle. 
More generally, 
\begin{prop}\label{prop:amenable_0_cosyst}
    Let $\cX$ be a finite, connected polygonal complex with $\pi_1(\cX,*)$ a residually finite amenable group. Then, for every $\eps>0$, there is a covering $\cY$ of $\cX$ such that $CoSyst_1(\cY,\Sym)\leq \eps$. 
\end{prop}
To that end, we need the following Theorem of Weiss.
\begin{thm}[Theorem 1 in \cite{weiss2001monotileable}]\label{thm:Weiss}
    Let $\Gamma$ be a residually finite amenable group, and let $S$ be a finite generating set of $\Gamma$. Then, for every $\eps>0$, there exists a  normal subgroup $N\trianglelefteq \Gamma$ of finite index $n$ and elements $A=\{a_1,...,a_n\}$ such that 
    \begin{enumerate}
        \item $A$ is a transversal of $\Gamma/N$, i.e., $AN=\Gamma$.
        \item $A$ is $\eps$-invariant with respect to left multiplication by $S$, i.e.,
\[
|SA\setminus A|=|\{sa\mid s\in S,a\in A, sa\notin A\}|\leq \eps |A|.
\]
    \end{enumerate}

\end{thm}
\begin{rem}
    This is not the formulation of Theorem 1 spelled out in \cite{weiss2001monotileable}, but, it is essentially what is proved therein.
\end{rem}
Moreover, we need the following characterization of the norm of a cocycle.
\begin{fact}\label{fact:levels_of_covering}
     Let $\alpha\colon \raE(\cX)\to \Sym(n)$ be a $1$-cocycle.  Let $\cY$ be the  covering of $\cX$ associated with $\alpha$ as in \eqref{eq:cocycle_to_covering}. Let $Level(i)=\{(x,i)\mid x\in V(\cX)\}\subseteq V(\cY)$ be the $i^{\rm th}$ \emph{level} of $\cY$, and let $\beta_i\colon V(\cY)\to \{0,1\}$ be its indicator function. Then
     \[
\|\alpha\|=\Ex_{\substack{e\sim \mu_1\\i\in[n]}} [\delta\beta_i(e,i)]
     \]
     Namely, $\|\alpha\|$ is the probability a sampled edge  has endpoints in different levels.\footnote{ This is an immediate consequence of the fact that $\delta\beta_i(e,i)$ indicates whether $\alpha(e).i\neq i$ or not.}
\end{fact}

\begin{proof}[Proof of Proposition \ref{prop:amenable_0_cosyst}]
   Without loss of generality, we can assume $\cX$ is a presentation complex of some finite presentation $\Gamma\cong \langle S|R\rangle$, namely it has a single vertex $*$. This is because, if we  retract a spanning tree $T$, the coverings of $\cX$ are in one to one correspondence with coverings of $\cX/T$. Now given a finite covering $\cY$ of $\cX$, denote by $\cY/T$ the retraction of all the lifts of $T$ to $\cY$.  Every cocycle on $\cY/T$ can be extended to  a cocycle on $\cY$ by sending all the edges of the lifts of $T$ to the identity. This operation only reduces the norm of the cocycle on $\cY$ compared to that on $\cY/T$, and thus upper bounding $CoSyst_1(\cY/T,\Sym)$ provides an upper bound for $CoSyst_1(\cY,\Sym)$.\footnote{Compare this to the proof of Proposition \ref{prop:tau_implies_edge_expansion}, in which we need to work harder. This is because there is no natural inequality in the reverse direction.} Furthermore, we can assume $\mu_1=\mu_S$ is uniform, otherwise we replace the soon to be chosen $\eps$ by $\frac{\eps}{|S|}$ and proceed as we do.

   Let $\eps>0$. By assumption, $\Gamma\cong \pi_1(\cX,*)$ is  residually finite and amenable. 
   By Theorem \ref{thm:Weiss}, there exists  $N\trianglelefteq \Gamma\cong \pi_1(\cX,*)$ a normal subgroup of index $n<\infty$ and $A=\{a_1,...,a_n\}$ such that $AN=\Gamma$ and $|SA\setminus A|\leq \eps|A|$. 
   Therefore, for every $h\in \Gamma$, $|SAh\setminus Ah|\leq \eps|Ah|=\eps|A|$ --- namely, the right translates of $A$ are also $\eps$-invariant.
   Since $\Gamma$ is residually finite, there exists a finite index subgroup $N'\subsetneq N$, with $[N:N']=k$. Choose a transversal $B=\{b_1,...,b_k\}\subseteq N$ of $N'$. 
   
   The left action of $\Gamma$ on $\Gamma/N$ induces a covering $\cY$ of $\cX=\cX_{\langle S|R\rangle}$, by letting $V(\cY)=A\cong\Gamma/N$ and $a_i\xrightarrow{s} a_j$ if and only if $sa_jN=a_iN$. Similarly, the left action of $\Gamma$ on $\Gamma/N'$ induces a connected covering $\cZ$ of $\cX$ with vertex set $A\times B\cong AB\cong \Gamma/N'$, and with $a_ib_{i'}\xrightarrow{s}a_jb_{j'}$ if and only if $sa_jb_{j'}N'=a_ib_{i'}N'$. Since $N'\lneq N$, the projection  to the first coordinate $A \times B\to A$ is a covering map from $\cZ$ to $\cY$. Define, as in Fact \ref{fact:levels_of_covering}, $Level_\cZ(b)=A\times \{b\}=\{(a,b)\mid a\in A\}$. This choice of levels is the same as choosing a $1$-cocycle of $\cY$ that represents the covering $\cZ\to \cY$ (as in the end of Section \ref{sec:prelim}). One can also read this cocycle directly as follows:   
   Since $\raE(\cY)$ can be parametrized by $S\times A$, the associated cocycle  $\alpha\colon S\times A\to \Sym(B)$ is defined by sending the pair $(s,a_i)$ to the following permutation of $B$ ---  embed $B$ in $\{a_i\}\times B$; then $s$ sends $\{a_i\}\times B$ to $\{s^{-1}a_i\}\times B$; by projecting to the second coordinate, we can extract the permutation on $B$. Hence, the levels of $\cZ$ with respect to $\alpha$ are $Level_\cZ(b)$.

   Now, if the edge labeled by $s$ who terminates in $(a,b)\in \cZ$ and whose  origin is at  $(a',b')$ where $b'\neq b$, then in particular $sab\notin Ab$. But we know that $SAb\setminus Ab$ is of size at most $\eps |A|$, which means there are at most $\eps |A||S|$ such edges. All in all, at most $\eps$ of the edges cross levels,   and by Fact \ref{fact:levels_of_covering}, $\|\alpha\|\leq \eps$ finishing the proof.

\end{proof}

Groups with  property $(\tau)$ (cf. \cite{lubotzky1994discrete}), a soft version of Kazhdan's property (T), stand as a contrast to amenable groups in many cases . This is  the case here as well. The main result of the  rest of this  section is that the finite coverings of a complex with fundamental group having property $(\tau)$, have a uniform lower bound on their  cosystoles. This gives also a new way of deducing large cosystoles in dimension $1$ with $\FF_2$-coefficients, providing an alternative to the methods used in \cites{kaufman2016isoperimetric,evra2016bounded}, as well as simplifying any future approach towards Problem \ref{prob:6.6.}.

\subsection{Edge expansion of coverings implies large cosystoles}

\begin{prop}\label{prop:edge_exp_implies_large_cosyst}
    Let $\cX$ be a polygonal complex. Let $\gamma>0$. Assume that  every finite connected covering $\cY$ of $\cX$ is an edge expander, namely, $h_0(\cY,\FF_2)\geq \gamma$ for some $\gamma>0$. Then  $CoSyst_1(\cX,\Sym)\geq \frac{\gamma}{2}$.
\end{prop}

\begin{proof}
    We begin with the setup of Fact \ref{fact:levels_of_covering}. Let $\alpha\colon \raE(\cX)\to \Sym(n)$ be a  connected $1$-cocycle which is not a coboundary (namely, a non-coboundary connected coycle). Let $\cY$ be the  covering of $\cX$ associated with $\alpha$ as in \eqref{eq:cocycle_to_covering}, with the measures on cells being uniformly lifted to $\cY$. Let $Level(i)=\{(x,i)\mid x\in V(\cX)\}\subseteq V(\cY)$ be the $i^{\rm th}$ \emph{level} of $\cY$, and let $\beta_i\colon V(\cY)\to \FF_2$ be its indicator function. The measure induced on $\cY$ is $\mu_j(\sigma,i)=\frac{1}{n}\mu_j(\sigma)$, where  $j=0,1$ or $2$, $\sigma\in \cX(j)$ and $i\in[n]$. 
    
    The idea of the proof is very simple: as  Fact \ref{fact:levels_of_covering} entails, $\|\alpha\|$ is proportional to the number of edges going between different levels. If $\cY$ is an expander, there are many such edges. Let us be more formal.
    We have, 
    \[
    \begin{split}
        \gamma \leq \frac{\Vert \delta \beta_i\Vert}{\underbrace{d(\beta_i,Z^0(\cY,\FF_2))}_{=\nicefrac{1}{n}}}
        =n\cdot \Ex_{(e,k)\sim \mu_1}[\delta\beta_i(e,k)]
        =\sum_{k=1}^n\Ex_{e\sim\mu_1}[\delta\beta_i(e,k)]=(\diamondsuit).
    \end{split}
\]
Now, $\delta\beta_i(e,k)=1$ when one of $(e,k)$'s endpoints is in $Level(i)$ and the other is not. Since the endpoints of $(e,k)$ are $(\tau(e),k)$ and $(\iota(e),\alpha(e).k)$, the condition is $k=i$ and $\alpha(e).k\neq i$ or $k\neq i$ and $\alpha(e).k= i$. Since every orientation of every edge is counted that way, one may count twice those of the first kind. Therefore,
\[
\begin{split}
    (\diamondsuit)=2\Ex_{e\sim\mu_1}[\delta\beta_i(e,i)]=2\Pro_{e\sim\mu_1}[\alpha(e).i\neq i].
\end{split}
\]
  On the other hand,
  \[
\Vert \alpha \Vert =\Pro_{\substack{e\sim \mu_1\\ i\in [n]}}[\alpha(e).i\neq i].
  \]
  Thus, by taking $\Ex_{i\in [n]}$ on $(\diamondsuit)$, we deduce $\Vert \alpha \Vert\geq \frac{\gamma}{2}$.
\end{proof}

\begin{cor}\label{cor:covers_have_latge_cosyst}
    Let  $\cX$ be a polygonal complex. Let $\gamma>0$. Assume that  every finite connected covering $\cY$ of $\cX$ is an edge expander, namely, $h_0(\cY,\FF_2)\geq \gamma$ for some $\gamma>0$. Then  $CoSyst_1(\cZ,\Sym)\geq \frac{\gamma}{2}$ for every \textbf{covering} $\cZ$ of $\cX$.
\end{cor}

\begin{proof}
    Since every covering of $\cZ$ is a covering of $\cX$, if the conditions of Proposition \ref{prop:edge_exp_implies_large_cosyst} are satisfied by $\cX$, then they are satisfied by $\cZ$.
\end{proof}


As we recalled in the beginning of Section \ref{sec:HDX}, the connected coverings of $\cX$ are in one to one correspondence with subgroups of the fundamental group $\pi_1(\cX,*)$. We will move now to show that if  $\pi_1(\cX,*)$ has property $(\tau)$, then all its finite coverings are edge expanding. 


\subsection{Property $(\tau)$ and edge expansion of coverings}\label{sec:prop_tau_and_edge_exp_of_coverings}

Let $\Gamma$ be a finitely generated group, and $S$ a finite generating set of $\Gamma$. Let $\mu$ be a probability distribution on $S$, viewed as a an element of the group ring $\mathbb{C}[\Gamma]$. Since every unitary representation $\rho$ of $\Gamma$ which acts on a Hilbert space $\mathbb{H}$ can be extended to the group ring,  we can define the operator $\rho(\mu)\in B(\mathbb{H})$. Namely, $\rho(\mu)=\sum_{s\in S} \mu(s)\rho(s)$. If there exists $\gamma>0$ such that, regardless of the chosen unitary representation $\rho$, the spectrum of $\rho(\mu)$ is contained in $[-1,1-\gamma]\cup\{1\}$, we say that $\mu$ has spectral gap $\gamma$. 
If this happens, then $\Gamma$ has Kazhdan's property (T), in which case  a spectral gap exists for any measure with a finite generating support. If the same is true only for finite dimensional representations   which factor through finite quotients of $\Gamma$, we say that $\Gamma$ has property $(\tau)$.

\begin{prop}\label{prop:tau_implies_edge_expansion}
    Let $\cX$ be a connected polygonal complex. If $\pi_1(\cX,*)$ has $(\tau)$, then there is a $\gamma=\gamma(\cX)>0$ such that for every connected covering $\cY$ of $\cX$ we have $h_0(\cY,\FF_2)\geq \gamma$.
\end{prop}

\begin{proof} [Proof sketch]
This proposition depends only on the $1$-dimensional skeleton of the complex, both in its statement and in the following proof. This is a generalization of the well known result (cf.\ \cite{lubotzky1994discrete}) that the Schreier graphs of groups with property $(\tau)$ are expanders. This exactly corresponds to the case where $\cX$ is a presentation complex, namely has a single vertex.

  Assume $*$ is a vertex of $G(\cX)$.   Given a spanning tree $T$ of $G(\cX)$, we can choose for every edge  $xy\in \raE(\cX)$ outside of $T$ an element in $\pi_1(\cX,*)$ as follows: Walk from $*$ to $x$ in $T$. Then traverse $xy$. Then walk from $y$ back to $*$. This is a closed path in $G(\cX)$ which originates at $*$ and hence it is an element of the fundamental group. Moreover, this is a generating set, and we  denote it by $S$. Take the uniform measure on $S$. By property $(\tau)$, there exists $\gamma_0>0$ such that the spectral gap of the operator $\Ex_{s\in S}[\rho(s)]$ for all finite dimensional unitary representations of $\pi_1(\cX,*)$ is at least $\gamma_0$.

      Let $\cY$ be a connected $n$-covering of $\cX$. Let $T_1,...,T_n$ be the lifts of the tree $T$ to $\cY$. We can retract the $T_i$'s in $\cY$ to points, resulting in a graph $\cZ$. Then, $\cZ$ is a Schreier graph of $\pi_1(\cX,*)$ with respect to the generating set $S$. Hence, the adjacency operator on $\cZ$ is of the form $\Ex_{s\in S}[\rho(s)]$, where $\rho$ is the action  of the fundamental group on the vertices of its Schreier graph. Thus, the adjacency operator on $\cZ$ has a spectral gap of at least $\gamma_0$. By the Cheeger inequalities, $\cZ$ is a good edge expander. 

      Now, $\cZ$ can be a great expander while $\cY$ is not. As an example, take $\cY=\cX$ and note that $\cZ$ has a single vertex and thus is an optimal expander, while $\cX$ may be as bad an expander as you would like. But, this is the \emph{only obstruction}. Namely, since we fix $\cX$ and look at coverings of it, there is a way of bounding the spectral gap of the covering $\cY$ by those of the retracted graph $\cZ$ and the base graph $\cX$ (or at least that of the tree $T$ that we chose). To see that, note that $\cY$ is essentially the replacement product of $\cZ$ with $T$. In this viewpoint, the trees $T_i$ are analogous to the clouds in the replacement product. Hence, given any function $\beta\colon V(\cY)\to \mathbb{C}$ which is perpendicular to the constant functions, we can average it out along the clouds. On the averaged function the adjacency of $\cY$ behaves similar to that of $\cZ$. And on the rest, we need to use the expansion of the tree $T$ (which again may be very poor, but it is a constant for all the coverings). From this argument we can deduce a spectral gap, which in turn bounds on the edge expansion  by (some variant of) the Cheeger inequalities. 

      Note that we totally ignored the weighting systems on the edges and vertices. The ideas we sketched still work, but we lose even more by the required change of weighting between $T$, $\cZ$ and $\cY$.
\end{proof}

  Proposition \ref{prop:tau_implies_edge_expansion} and Corollary \ref{cor:covers_have_latge_cosyst} prove that given a fixed complex with a property $(\tau)$ fundamental group, the family of all its covers have a uniform lower bound on their cosystoles ---  with coefficients in permutations, and hence, in particular, also for $\FF_2$-coefficients. This proves clause $(2)$ of Theorem \ref{thm:intro_exp_of_coverings_implies_cosystols}. Since there are many groups with property $(\tau)$, this provides a plethora of examples. In particular, this applies to all finite sheeted coverings of a complex whose fundamental group is a lattice in high rank simple Lie groups.

\subsection{Local to Global expansion}\label{sec:local_global_exp}
Since the bound in Proposition \ref{prop:tau_implies_edge_expansion} was presented in a qualitative manner --- we intentionally excluded  calculations from the proof sketch we provided --- we opt for a different approach which yields  quantitative results. This approach requires one to assume more about the complex in hand than just having a fundamental group with property $(\tau)$. It needs to be a \emph{local spectral expander}, which in turn implies property (T), and thus $(\tau)$, of the fundamental group (this is Garland's method \cite{garland1973p}, see also \cites{ollivier2003spectral,ZukPropT,ballmann1997l2}).  
In most applications, and specifically in the Ramanujan complexes of Lubotzky--Samuels--Vishne \cite{LSV2005ramanujan} which were used in \cites{kaufman2016isoperimetric,evra2016bounded}, this condition holds. Moreover, when studying a fixed simple $p$-adic Lie group $G$ with an associated Bruhat--Tits building $\cB$, this method provides a lower bound on the cosystoles of all quotients of $\cB$ uniformly, independent of the acting lattice.
To be able to present the local spectral expansion condition, we need to recall some spectral definitions of expansion. We follow  \cite{Dikstein_lecture_notes_trickling_down} in our presentation --- See also Appendix A of \cite{harsha2022note}, and the original paper by Oppenheim for the  case where the distribution on highest dimensional cells is uniform \cite{oppenheim2018local}. 

Let $\cX$ be a connected $2$-dimensional \textbf{simplicial} complex. Assume $\mu_2$ is a fully supported distribution on the triangles of $\cX$. Furthermore, assume $\mu_1$ and $\mu_0$ are descendent from $\mu_2$, and are fully supported. This implies in particular that the complex $\cX$ is \emph{pure}: Every edge (and vertex) participate in some triangle. The link of a cell  $\sigma\in \cX$ is the simplicial complex
\[
\cX_\sigma=\{\tau\in \cX\mid \tau\cup \sigma\in \cX, \tau \cap \sigma=\emptyset\}.
\]
In the two dimensional case, the only links we will focus on are vertex links, which are graphs. The links inherit a distribution on their edges as follows: For every vertex $\sigma$ and edge $e\in \cX_\sigma$, we have $\mu_{\sigma,1}(e)=\frac{\mu_2(\sigma\cup e)}{\sum_{e'\in \cX_\sigma(1)}\mu_2(\sigma \cup e')}$. The measure on vertices of $\cX_\sigma$ descends from $\mu_{\sigma,1}$. Given a simplical complex $\cX$, we can define an adjacency operator $A$ on the complex valued functions on its vertices: Given $f\colon \cX(0)\to \mathbb{C}$, $$Af(v)=\Ex_{u\sim \mu_{v,0}}[f(u)],$$
namely, we sample an edge $vu$ according to $\mu_1$ conditioned on the edge containing $v$. Furthermore, we choose the following inner product on $\mathbb{C}^{\cX(0)}$:
\[
\langle f,g\rangle=\Ex_{v\sim \mu_0}[\overline{f(v)}g(v)].
\]
The adjacency operator $A$ has the following properties:
\begin{itemize}
    \item The constant functions are eigenfunctions of it with eigenvalue $1$.
    \item It is self adjoint with respect to the chosen inner product.
\end{itemize}
We say that $\cX$ is a $\lambda$-spectral expander if the second eigenvalue of $A$ is bounded by $\lambda$ from above. We say that $\cX$ is a $\lambda$-local spectral expander if the links $\cX_v$ are $\lambda$-spectral expanders for all $v\in V(\cX)$.
\begin{thm}[Oppenheim's trickling down theorem, \cites{Dikstein_lecture_notes_trickling_down,oppenheim2018local}]\label{thm:trickle}
   Let $\cX$ be a pure, connected, $2$-dimensional simplicial complex with fully supported descendent measures on its cells.  Assume also that $\cX$ is a $\lambda$-local spectral expander for some $\lambda<\frac{1}{2}$ --- namely, that for each vertex $v\in V(\cX)$, the link $\cX_v$ is a $\lambda$-spectral expander. Then $\cX$ is a $\frac{\lambda
}{1-\lambda}$-spectral expander.
\end{thm}

\begin{prop}[Weighted Cheeger inequalities: The lower bound]\label{prop:weighted_Cheegr_lower_bound}
    If $\cX$ is a $\lambda$-spectral expander, then $h_0(\cX,\FF_2)\geq 1-\lambda$.
\end{prop}
Though the proof method of  Proposition \ref{prop:weighted_Cheegr_lower_bound} is standard,  we did not find a source which proves it in this generality, and decided to indclude it in Appendix \ref{appendix:weighted_Cheeger} for completeness. 
\begin{cor}\label{cor:local_exp_implies_exp_of_coverings}
    Let $\cX$ be a $\lambda$-local spectral expander for some $\lambda<\frac{1}{2}$. Then, for every connected covering $\cY$ of $\cX$, we have both $h_0(\cY,\FF_2)\geq \frac{1-2\lambda}{1-\lambda}$,  and  $CoSyst_1(\cY,\Sym)\geq \frac{1-2\lambda}{2-2\lambda}$.
\end{cor}
\begin{proof}
    Since links are preserved by coverings, $\cY$ is a $\lambda$-local spectral expander. Since $\cY$ is connected as well, it satisfies the conditions of Theorem \ref{thm:trickle}, and thus it is a $\frac{\lambda}{1-\lambda}$-spectral expander. By Proposition \ref{prop:weighted_Cheegr_lower_bound}, we deduce that $h_0(\cY,\FF_2)\geq 1-\frac{\lambda}{1-\lambda}=\frac{1-2\lambda}{1-\lambda}$. The second claim follows now from Corollary \ref{cor:covers_have_latge_cosyst}.
\end{proof}
As mentioned in the introduction, Corollary \ref{cor:local_exp_implies_exp_of_coverings} applies uniformly to all the finite quotients of a Bruhat--Tits building, provided that it is a $\lambda$-local spectral expander with $\lambda<\frac{1}{2}$ --- which, by choosing $k=0$ in Theorem 4.1 of \cite{dikstein2023coboundary}, is typically the case.  In particular, this proves the uniform lower bound on large cosystoles in \cites{kaufman2016isoperimetric,evra2016bounded}, in a different manner, at least for dimension $1$.

\subsection{Codimension one systole of Riemannian manifolds}\label{sec:Riemannian_man}

In this section we will give an unexpected application of our study.  Let $(M_0,g_0)$ be a closed Riemannian manifold of dimension $n$, equipped with a fixed triangulation $\cX_0$. For a finite sheeted covering $M\to M_0$, let $g$ (respectively $\cX$) be the pullback of  $g_0$ (respectively $\cX_0$) to $M$. The $i$-dimensional systole of $(M,g)$ with $\FF_2$ coefficients is the infimum over the volumes of homologically non-trivial Lipschitz $i$-cycles in $(M,g)$.
We denote this value by $Sys_i(M,\FF_2)$. Let us define it with more details  (cf.\ Section 2 in \cite{guth2014quantum}). First, a Lipschitz $i$-chain $\alpha$ with coefficients in $\FF_2$ is a finite sum $\sum_{j=1}^t a_jf_j$, with $a_j\in \FF_2$ and $f_j$  Lipschitz maps from the standard $i$-simplex $\Delta^i$ to $M$. Denote $\vol(\alpha)=\sum a_j\vol(f_j)$, where $\vol(f)$ for $f\colon \Delta^i\to (M,g)$ is the volume of the pullback metric $f^*(g)$ on $\Delta^i$. Let $C_i=C_{i,Lip}(M,\FF_2)$ be the space of $i$-chains and $\partial_i=\partial_{i,Lip}\colon C_i\to C_{i-1}$ the $i$-boundary map, which is defined by restricting $f_j$ to the $(i-1)$-skeleton of $\Delta^i$. Let $Z_i=Z_i(M,\FF_2)$, the $i$-cycles, be the kernel of $\partial_i$, and let $B_i=B_i(M,\FF_2)$, the $i$-boundaries, be the image of $\partial_{i+1}$. Finally, let $H_i=H_i(M,\FF_2)$ be the $i^{\rm th}$ homology, namely $Z_i/B_i$. So, $Sys_i(M,g)$ is the infimum over $\alpha\in Z_i\setminus B_i$ of $\vol(\alpha)$. If $H_i=0$, then we define $Sys_i(M,g)=\infty$.

\begin{thm}\label{thm:Riemannian_manifolds}
    Let $(M_0,g_0)$ be a closed Riemannian manifold of dimension $n$.  Assume $\pi_1(M_0,*)$ has property $(\tau)$. Then, there exists a constant $c=c(M_0,g_0)>0$ such that for every $M$ a finite sheeted covering of $(M_0,g_0)$ equipped with the pullback metric $g$,  $$Sys_{n-1}(M,\FF_2)\geq c\cdot \vol(M).$$
\end{thm}

This applies, for example, to all manifolds $M_0={}_{\Gamma_0}\backslash^ G/_K$, when $G$ is a high-rank simple real Lie group, $K$ a maximal compact subgroup and $\Gamma_0$ a torsion free cocompact lattice in $G$. Here $n=\dim \nicefrac{G}{K}$, where $\nicefrac{G}{K}$ is the symmetric space associated with $G$, and so $M$ is a locally symmetric closed manifold.  It is interesting to compare Theorem \ref{thm:Riemannian_manifolds} with the following result.
\begin{prop}\label{prop:log_systole}
    Let $M_0,g_0$ and $M$ be as in Theorem \ref{thm:Riemannian_manifolds}. Then there exists a constant $c'=c'(M_0,g_0)>0$ such that $Sys_1(M)\leq c'\cdot\log(\vol(M))$.
\end{prop}
We postpone the proof of Proposition \ref{prop:log_systole} to the end of this section.
Let us add a remark before we sketch the proof of Theorem \ref{thm:Riemannian_manifolds}. In principal, it might be that $M_0$ has no finite sheeted coverings (i.e., $\pi_1(M_0,*)$ has no finite index subgroups) or even if it has many (e.g., $\pi_1(M_0,*)$ is an infinite residually finite group), it might be that for all coverings $H_{n-1}(M,g)=0$. 
In these cases, Theorem \ref{thm:Riemannian_manifolds} says nothing. Still, for the most interesting examples, when $\Gamma_0$ as above is a cocompact lattice in a high-rank simple Lie group, there are always infinitely many coverings with non-trivial $n-1$ homology with $\FF_2$ coefficients. This is deduced by applying Poincar\'e duality (cf. Chapter 3.3 of \cite{Hatcher_Alg_Top}) on the results of \cite{lubotzky1987finite}.

\begin{proof}[Proof sketch of Theorem \ref{thm:Riemannian_manifolds}]

It is well known that if $\cX$ is the triangulation of $M$ as above, then $H_i(\cX,\FF_2)=H_i(M,\FF_2)$ (cf. Chapter 2.1 of \cite{Hatcher_Alg_Top}). Furthermore, if $\alpha$ is an $i$-cycle of $\cX$, it can be considered as a Lipschitz cycle of $M$. 
Moreover, as shown in Lemma 4 in \cite{guth2014quantum}, the norm $\|\alpha\|$ of $\alpha$ as an element of $Z_i(\cX,\FF_2)$ is at least $c_0(M_0,g_0,\cX_0)\vol(\alpha)$ --- note that when evaluating $\vol(\alpha)$ we consider it as an element of $Z_i(M,\FF_2)$.
Though it is not stated explicitly in \cite{guth2014quantum}, there is also an inverse inequality in the following sense: 
Every $i$-cycle $\alpha$ in $M$ can be approximated in $\cX$ by ``pushing'' $\alpha$ to the $i$-skeleton of $\cX$, resulting with an $i$-cycle $\overline \alpha$ of $\cX$. For this cycle we have $\|\overline \alpha\|\leq c_1(M_0,g_0,\cX_0)\vol(\alpha)$ for some $c_1>0$. This can be deduced from a result of Federer--Fleming; See \cite{guth2006notes} for the Euclidean case, although the argument is general and applies in our more general setup. 

All in all, $Sys_i(M,\FF_2)$ and $Sys_i(\cX,\FF_2)$ bound each other up to a constant multiple. Furthermore, as explained in Section 2 of \cite{guth2014quantum}, similar arguments apply for the cosystoles, and in particular, by using the Poincar\'e dual polyhedral  $\cX'$ of $\cX$, one can deduce that $Sys_{n-1}(M,\FF_2)\geq c\cdot CoSys_1(\cX,\FF_2)$ for some $c>0$ that depends only on $M_0,g_0$ and $\cX_0$. Thus, clause $(2)$ of Theorem \ref{thm:intro_exp_of_coverings_implies_cosystols}, which was proved in Section \ref{sec:prop_tau_and_edge_exp_of_coverings}, finishes the proof.

\end{proof}

We learned the following Lemma from  Larry Guth:
\begin{lem}\label{lem:Guth}
    Let $(M,g)$ be a closed Riemannian manifold with a non-trivial first $\FF_2$-homology. Then $Sys_1(M,\FF_2)\leq 2\textrm{Diam}(M)$, where $\textrm{Diam}(M)$ is the diameter of $M$, namely, the length of the shortest path between the furthest two points on the manifold.
\end{lem}
\begin{proof}
    Let $\sigma$ be a shortest representative of a non-trivial element in $H_1(M,\FF_2)$. Every two points $x,y$ on $\sigma$ separate it into two paths $x\xrightarrow{\sigma_1}y$ and $x\xrightarrow{\sigma_2}y$ such that $\sigma=\sigma_1\sigma_2^{-1}$. Let $x\xrightarrow{\sigma_3}y$ be a shortest path in $M$ between $x$ and $y$. Then $\sigma_1\sigma_3^{-1}+\sigma_3\sigma_2^{-1}=\sigma$ as elements of $H_1(M,\FF_2)$. Thus, either $\sigma_1\sigma_3^{-1}$ or $\sigma_3\sigma_2^{-1}$ are non-trivial in $H_1(M,\FF_2)$, and by our assumption on $\sigma$ as a shortest representative, we know that $$\max\{\ell(\sigma_1\sigma_3^{-1}),\ell(\sigma_3\sigma_2^{-1})\}\geq \ell(\sigma).$$ Since we can choose $x$ and $y$ such that $\ell(\sigma_1)=\ell(\sigma_2)=\frac{1}{2}\ell(\sigma)$, and by the definition of the diameter, $\ell(\sigma_3)\leq \textrm{Diam}(M)$, we can deduce that for this choice of $x$ and $y$ 
    \[
\max\{\ell(\sigma_1\sigma_3^{-1}),\ell(\sigma_3\sigma_2^{-1})\}\leq \frac{\ell(\sigma)}{2}+\textrm{Diam}(M).
    \]
    Therefore $\ell(\sigma)\leq 2\textrm{Diam}(M)$ and we deduce the Lemma.
\end{proof}

\begin{proof}[Proof of Proposition \ref{prop:log_systole}]
Recall that $\cX$ is the pullback of the triangulation $\cX_0$ of $M_0$ to $M$.
By Proposition \ref{prop:tau_implies_edge_expansion}, the $1$-skeleton of $\cX$ is an expander graph, and hence $\textrm{Diam}(\cX)$ is $O(\log|V(\cX)|)$. Since $|V(\cX)\|$ is proportional (up to multiplicative constants) to $\vol(M)$, and $\textrm{Diam}(X)$ is proportional to $\textrm{Diam}(M)$, the conclusion follows from Lemma \ref{lem:Guth}.
\end{proof}

\section{\textbf{Bounded degree $2$-dimensional $\FF_2$-coboundary exapnders}}\label{sec:two_dim_coboundary_expanders_F2}
In this section we resolve Problem \ref{prob:coboundary_expanders_F2} for the special case of dimension $2$. As oppose to Section \ref{sec:large_cosystols}, where we deduce results about $\FF_2$ coefficients by using the permutation coefficients framework, here we use a direct approach, relying heavily on the work of Evra--Kaufman \cite{evra2016bounded}. This section requires some expertise in the theory of arithmetic and algebraic groups, but it is not needed for the rest of the paper.

Let $G$ be a simple $p$-adic Lie group of rank $d\geq 3$, and let $\cB$ be the Bruhat--Tits building of $G$. The building $\cB$ is a contractible, infinite, simplicial complex of dimension $d$ which has a bounded degree --- see \cite{tits1979reductive} or Section 3.4 in \cite{platonov1993algebraic} for the definitions and some background. For every cocompact lattice $\Gamma$ of $G$, the simplicial complex $\cX= {_\Gamma}\backslash^\cB$ is finite and of degree bounded by the degree of $\cB$ (which is independent of $\Gamma$). 
The following is a corollary from the main theorem  of \cite{evra2016bounded}:
\begin{thm}[Corollary 1.12 in \cite{evra2016bounded}]\label{thm:cor_of_EK}
 For a large enough prime number $p$,  $h_1(\cX,\FF_2)\geq \lambda>0$.
\end{thm}

Therefore, to prove the existence of bounded $2$-dimensional coboundary expanders, it suffices  to prove that for \textbf{some} $G$ as above and infinitely many lattices $\Gamma$ of $G$, the first cohomology of $\cX={_\Gamma}\backslash^\cB$ with $\FF_2$ coefficients vanishes. Now, as $\cB$ is contractible, $H^1({_\Gamma}\backslash^\cB, \mathbb{F}_2) = H^1(\Gamma, \mathbb{F}_2)$.  But, $H^1(\Gamma, \mathbb{F}_2) \cong \nicefrac{\Gamma}{[\Gamma,\Gamma]\Gamma^2}$. Hence: 
\begin{fact}\label{fact:vanishing_of_cohomology}
    The vanishing of $H^1(\cX, \mathbb{F}_2)$ is equivalent to the claim that $\Gamma$ does not have a quotient group of order 2. Equivalently, $\widehat\Gamma$ the profinite completion of $\Gamma$ does not have such a quotient.
\end{fact} 
The rest of the section is devoted to an appropriate choice of $G$ and  infinitely many lattices without quotients of order $2$. Note that such lattices will provide us simplicial complexes of dimension $d=\textrm{rank}(G)$, but we care only about the $2$-skeleton of these complexes.

Let $k$ be a number field and $\frak{G}$ a $k$-algebraic group which is simple, connected and simply connected. Let $\mathbb{A}_k$ be the ring of Adeles of $k$, i.e., $\mathbb{A}_k=\prod^*_{\nu} k_\nu$ where $\nu$ runs over all norms of $k$ and $k_\nu$ is the completion of $k$ with respect to $\nu$. Let $\mathbb{A}_k^\infty$ (respectively $\mathbb{A}_k^f$) be the product over all the archemedian norms (respectively non-archemedian norms). Now, $\frak{G}(k)$ is a discrete subgroup embedded diagonally in $\frak{G}(\mathbb{A}_k)$, in fact, a lattice.
By the strong approximation theorem (SAT for short, see Section 7.4 in \cite{platonov1993algebraic}), if for some $\nu_0$, $\frak{G}(k_{\nu_0})$ is not compact, then $\frak{G}(k)\times \frak{G}(k_{\nu_0})$ is dense in $\frak{G}(\mathbb{A}_k)$. Equivalently, $\frak{G}(k)$ is dense in $\frak{G}(\mathbb{A}_{k\setminus \{\nu_0\}})$, where $\mathbb{A}_{k\setminus \{\nu_0\}}=\prod^*_{\nu\neq \nu_0} k_\nu$.

Assume from now on that $\frak{G}(\mathbb{A}_k^\infty)$ is compact. Fix a non-archemedian norm $\nu_0$ where $\nu_0\not|\ 2$ --- i.e., the prime associated with $\nu_0$ does not lie over the rational prime $2$ --- so that $\frak{G}(k_{\nu_0})$ is \textbf{not} compact. Now, for every  open compact subgroup $K$ of $\frak{G}(\mathbb{A}_{k\setminus\{\nu_0,\infty\}})$, $\Gamma=\frak{G}(k)\cap K$ is a dense subgroup of $K$ (by SAT), but a discrete subgroup, in fact a lattice, in $\frak{G}(k_{\nu_0})$. Assume further that $\Gamma$ satisfies the congruence subgroup property (CSP for short, see Section 9.5 in \cite{platonov1993algebraic}). I.e., every finite index subgroup $\Gamma_0$ of $\Gamma$ contains a subgroup of the form $G(K)\cap K_0$, where $K_0$ is an open compact subgroup of $K$. Combining with SAT, this implies that $\widehat \Gamma\cong K$.

\begin{prop}\label{prop:open_compact_at_2_implies_vanishing}
    With the above assumptions; if $\frak{G}(k_2):=\prod_{\nu|2}\frak{G}(k_\nu)$ has an open compact subgroup $K_2$ with $H^1(K_2,\FF_2)=0$, then $\frak{G}(k_{\nu_0})$ has infinitely many cocompact (torsion free) lattices $\Gamma_i$ with $H^1(\Gamma_i,\FF_2)=0$.
\end{prop}

\begin{proof}
    In the notations above, choose $K=K_2\times K'$, where $K'$ is an open compact subgroup of $\frak{G}(\mathbb{A}_{k\setminus\{2,\nu_0,\infty\}})$ of the form
     \[
    K'=\prod_{\nu\notin\{\infty, 2,\nu_0\}}K_\nu,
    \]
    and each $K_\nu$ in the product is an open compact subgroup of $\frak{G}(k_{\nu})$.
    Let $\Gamma = \frak{G}(k)\cap K$, so $\Gamma$ is a lattice in $\frak{G}(k_{\nu_0})$ and $\widehat \Gamma\cong K$. 

    Now, $\Gamma$ being a cocompact lattice in $\frak{G}(k_{\nu_0})$ is finitely generated, and thus $K$ is a finitely generated profinite group. Hence, $\dim H^1(K,\FF_2)=\dim(\textrm{Hom}(K,\FF_2))$ is finite, i.e., for only finitely many $\nu$'s it can happen that $H^1(K_\nu,\FF_2)\neq 0$. Each of these $K_\nu$'s is a virtually (torsion free) pro-$p$ group, where $p$ is some odd prime. Replace $K_\nu$ by a deep enough finite index subgroup which is (torsion free) pro-$p$ to ensure that  $H^1(K_\nu,\FF_2)= 0$. Now $H^1(K,\FF_2)= 0$. Fixing some $\nu'\notin\{\infty, 2, \nu_0\}$ which lies over an odd rational prime $\ell$,  choose a sequence of decreasing pro-$\ell$ subgroups $K_{\nu'}(i)$ of $K_{\nu'}$ and let 
    \[
        K(i)=K_2\times K_{\nu'}(i)\times \prod_{\nu\notin\{\infty, 2,\nu_0,\nu'\}}K_\nu.
    \]
   All in all, if $\Gamma(i)=K(i)\cap \frak{G}(k)$, they all satisfy $\widehat {\Gamma(i)}=K(i)$, and their first cohomology with $\FF_2$ coefficients vanishes.
\end{proof}

Let us now construct an example that satisfies all the needed assumptions\footnote{This example will be used again in Section \ref{sec:cocyc_exp_lattices_sofic}}. Let $D$ be a quaternion algebra over $\mathbb{Q}$ which ramifies at $\infty$ and at some rational prime $p_0\neq 2$ (cf. \cite{vigneras1980algebres} or \cite{platonov1993algebraic}). Let $\tau$ be the standard involution of $D$ and $h\colon D^n\times D^n\to D$ the canonical sesqui-linear form on $D^n$, i.e., 
\[
h((x_1,...,x_n),(y_1,...,y_n))=\sum_{i=1}^n x_i \tau (y_i).
\]
Let $\frak{G}$ be the algebraic group $SU(n,D,h)$, i.e., the $D$-linear transformations preserving $h$. This is a $\QQ$ algebraic group which is simple, connected and simply connected, of type $C_n$, i.e., it is a form of $Sp(2n)$. See the discussion in Section 5 of  \cite{DGLT} (note that there $p_0=2$, but all the results are the same, till the last point), as well as in \cite{Ghola}, and also Section 2.3 in \cite{platonov1993algebraic}.
If $p\neq p_0,\infty$, then $\frak{G}$ splits over $\mathbb{Q}_p$ and $\frak{G}(\QQ_p)\cong Sp(2n,\QQ_p)$. In particular, for $p=2$, $\frak{G}(\QQ_2)$ contains an open compact subgroup $K_2=Sp(2n,\ZZ_2)$.
For $n\geq 3$,   $Sp(2n,\ZZ_2)$ is a perfect group.\footnote{This follows from the well known fact that its dense subgroup $Sp(2n,\ZZ)$ is perfect. In fact, for both it can be read directly from their Steinberg presentation (cf.\ \cite{digne2022braid}).} 
 Hence, it cannot have abelian quotients,
 and in particular $H^1(K_2,\FF_2)=0$. In addition, for every $p\neq p_0$, we get as before many arithmetic lattices $\Gamma$ in $\frak{G}(\QQ_p)$. By \cite{rapinchuk1989congruence} and \cite{tomanov1989congruence} all these $\Gamma$'s satisfy the CSP. Using Proposition \ref{prop:open_compact_at_2_implies_vanishing} and Theorem \ref{thm:cor_of_EK} , we constructed an infinite families of bounded degree cocycle expanders with vanishing first cohomology, which are in turn $1$-coboundary expanders with $\FF_2$ coefficents, resolving Problem \ref{prob:coboundary_expanders_F2} in dimension $2$.

\begin{rem}
    \ 
    \begin{enumerate}
        \item One can deduce the following from the proof: If $G$ is a simple Lie group with one arithmetic lattice $\Gamma$ satisfying CSP and $H^1(\Gamma,\FF_2)=0,$ then $G$ has infinitely many lattices with covolume tending to infinity and vanishing first cohomology with $\FF_2$ coefficients. 
        \item The previous clause should be compared with the main results of \cite{lubotzky1987finite}, which asserts that every finitely generated not solvable linear group and in particular the above $\Gamma$ has a sequence of finite index subgroups $\Gamma^j$ with $\dim H^1(\Gamma^j,\FF_2)\xrightarrow{j\to \infty}\infty.$
        \item In the above --- both in Proposition \ref{prop:open_compact_at_2_implies_vanishing} and the previous two clauses --- we can replace $2$ by any fixed rational prime $\ell$.
    \end{enumerate}
\end{rem}

\section{\textbf{Open problems and suggestions for further research}} \label{sec:open_problems}

\subsection{\textbf{Cocycle expansion in permutations of  $p$-adic lattices implies the existence of non-sofic groups}} 
\label{sec:cocyc_exp_lattices_sofic}

In this paper we are following the footsteps of the theory developed for coboundary expansion over $\FF_2$ with the goal to generalize it to $\Sym$. In Section \ref{sec:Examples}, we analyzed the complete complexes and spherical buildings. In Section \ref{sec:large_cosystols}, we studied the large cosystoles condition with permutation coefficients, and related it to property $(\tau)$ and Garland's method. In particular, our analysis applies to quotients of Bruhat--Tits buildings of simple $p$-adic Lie gropus of high rank. The natural next step would be to try to prove \textbf{cocycle} expansion with permutation coefficients for such quotients.
Following the foot steps of \cites{kaufman2016isoperimetric,evra2016bounded,dikstein2023coboundary}, one may dare to suggest:
\begin{conj}\label{conj:Sym_mimics_F2}
    Let $G$ be  simple $p$-adic Lie group of rank $d\geq 3$, and $p\geq p(d)\gg0$. Then, for  any lattice $\Gamma$   of $G$,
    \begin{equation}\label{eq:conj_sym_F2}
    h_1({}_\Gamma\setminus^\cB,\Sym)>0,    
    \end{equation}
    where $\cB$ is the Bruat--Tits building associated with $G$.
\end{conj}
Unfortunately, we are currently not able to prove the above. Proving such a result, even for very special cases can be  very significant. Recall:

\begin{defn}
    A finitely presented group $\Gamma\cong \langle S|R\rangle$ is said to be \emph{sofic} if there is a sequence of functions $f_n\colon S\to \Sym(n)$,  such that 
    \[
      \forall r\in R\ \colon \ \ d_h(f_n(r),\Id)\xrightarrow{n\to \infty}0,
    \]
    and for every $w\notin \langle\langle R\rangle \rangle$,\footnote{$\langle\langle R\rangle \rangle$ is the normal subgroup generated by $R$.}
    \[
        \liminf(d_h(f_n(w),\Id))\geq \frac{1}{2}.
    \]
\end{defn}
\begin{problem}[Gromov \cite{Gromov_sofic}, Weiss \cite{Weiss_Sofic}] \label{prob:sofic_groups}
    Are there non-sofic groups?
\end{problem}

\begin{thm}\label{thm:Gohla_Thom}
   Conjecture \ref{conj:Sym_mimics_F2} implies the existence of non-sofic groups
\end{thm}

\begin{proof}
    Recall that if $h_1({}_\Gamma\setminus^\cB,\Sym)>0,$ then  $\Gamma$ is homomorphism stable with linear rate (see Theorem 1.3 in  \cite{CL_part1}) --- which is known to imply  flexible stability in permutations. 

    Let $D$ be a quaternion algebra over $\QQ$, which ramifies at infinity and some rational prime $p_0$, as in Section \ref{sec:two_dim_coboundary_expanders_F2}. This time, we may take $p_0=2$. Let $\Gamma_0$ be the lattice in $Sp(2n,\QQ_p)$ obtained from such a $D$ as in Section \ref{sec:two_dim_coboundary_expanders_F2}, where $p\neq p_0$ (see \cite{DGLT}, where such $\Gamma_0$ is constructed with $p_0=2$).
    As in \cite{DGLT}, by adapting Deligne's method to the $p$-adic numbers, $\Gamma_0$ has a central extension $\tilde \Gamma_0$ with finite center and $\tilde \Gamma_0$ is not residually finite. 
    Now, in a paper in preparation of Gohla--Thom \cite{Gohla_Thom} (see also Corollary 3.4.5 in \cite{Ghola}), they show that for $n\geq 4$, if $\Gamma_0$ is flexibly stable, then $\tilde \Gamma_0$ is a non-sofic group.
\end{proof}

We should emphasize, that what we need is much weaker than the full force of Conjecture \ref{conj:Sym_mimics_F2}. First, we need it only for the group $G=Sp(2n,\mathbb{Q}_p)$, where $n$ is some fixed number larger than $4$, and $p$ a fixed prime, large enough with respect to $n$ (see \cite{Gohla_Thom}). Even then, it suffices to prove \eqref{eq:conj_sym_F2} for a single lattice $\Gamma_0$ and not all of them. Moreover, we do not need a linear homomorphism stability rate but any rate would suffice. 

Let us put our observation in perspective: The existence of non-sofic groups is a long standing problem, and many examples of groups have been suggested as candidates for being non-sofic. For example, finitely generated infinite simple groups \cite{burger2000lattices}, the Burnside groups, the Higman group (cf. Proposition 6 in Section 1.1 of \cite{Trees_Serre}) and others. 
But, as of now, these efforts have not led to a path towards a proof \cites{thom2018finitary,helfgott2019soficity,kassabov2019soficity}.

The work of Bowen--Burton \cite{BowenBurton} suggested a new path. They showed that if $SL_n(\ZZ)$, for $n\geq 5$, is homomorphism stable then there exists a non-sofic group. There are good reasons to believe that their methodology can work for many high rank arithmetic groups other than $SL_n(\ZZ)$, including the lattices discussed in \cite{Gohla_Thom}. 
The properties of $SL_n(\ZZ)$ used in \cite{BowenBurton} are: 
\begin{enumerate}
    \item The congruence subgroup property (CSP).
    \item Existence of two copies of the same non-abelian free group, one  Zariski-dense in $SL_n(\ZZ)$, and the other Zariski-dense in a proper simple algebraic subgroup. 
    \item Super strong approximation of $SL_n(\ZZ)$ with respect to \textbf{all} its congruence subgroups.
\end{enumerate}
Properties $(1)$ and $(2)$ hold for  many other arithmetic groups $\Gamma$, but $(3)$ is not yet known. An ongoing work of Alireza Salehi Golsefidi and others aims to close this gap. Anyway, for our goal the main missing point is to show that $\Gamma$ is homomorphism stable.

Gohla--Thom \cite{Gohla_Thom} take a somewhat different path: Let $\Gamma_0$ be as in the proof of Theorem \ref{thm:Gohla_Thom} (or one of the lattices constructed in Section \ref{sec:two_dim_coboundary_expanders_F2}).  In \cite{DGLT}, the stability of homomorphisms of $\Gamma$ into the unitary group with the Frobenius norm, which was deduced from cohomology vanishing results of Garland \cite{garland1973p}, was used to deduce that $\tilde \Gamma_0$ is not Frobenius approximated (which is  analogous to non-soficity). By imitating this pattern, \cite{Gohla_Thom}  prove that if $\Gamma_0$ is homomorphism stable in permutations then $\tilde \Gamma_0$ is non-sofic.  

Our observation here suggests a technical tool to prove the missing homomorphism stability of this $\Gamma_0$: In \cite{evra2016bounded}, the $\FF_2$ analogue of Conjecture \ref{conj:Sym_mimics_F2} was proved. Furthermore, the version of Conjecture \ref{conj:Sym_mimics_F2} with permutation coefficients, but with respect to the discrete distance instead of the Hamming one, was proved in \cite{dikstein2023coboundary}. We hope that the techniques developed in these papers may lead to a positive resolution of Conjecture \ref{conj:Sym_mimics_F2}, if not in general, then at least for some specific cases, such as $\Gamma_0$, which imply the existence of non-sofic groups.

Although we are optimistic, there are several warning signs along the way which one should consider. First, in \cite{BeckerLubotzky}, it was shown that property $(T)$ groups never have a strict version of homomorhpism stability. Furthermore,   in \cite{ioana2020cohomological} it was shown that  $SL_n(\ZZ)$, where $n\geq 3$, does not have the local lifting property, which is another strngthening of homomorphism stability into unitaries with the Hilbert--Schmidt norm. Lastly, in \cite{glebsky2022extensions} it is shown that the central extension $\tilde \Gamma_0$ above   \textbf{is}  weakly sofic,  which is, as it sounds, a weakening of the notion of a sofic group. This means that our cnadidates for being non-sofic are actually weakly sofic!

\subsection{Bounded expanders in permutations}\label{sec:generalized_cosystolic_expansion}

The following problems are the natural generalizations of the search for bounded coboundary and cocycle expanders with $\FF_2$ coefficients (Problem \ref{prob:coboundary_expanders_F2}) to the permutation setup. Namely, we are looking for infinite families of connected finite polygonal complexes $\{\cX_m\}_{m=1}^\infty$, which are uniformly bounded, namely there is a positive integer $k$ such that the degree of every vertex in $\cX_m$ is bounded by $k$.
 
\begin{problem}[Bounded coboundary expanders in permutations] \label{problem:Bounded_cobundary_exp_in_Sym}
    Is there such a family for which 
    \[
\forall m\in \mathbb{N}\ \colon \ \ 
 h^B_0(\cX_m,\Sym),h^B_1(\cX_m,\Sym)\ge \lambda,
    \]
    for some $\lambda>0$ independent of $m$?
\end{problem}

This is the permutation analogue of what was constructed in Section \ref{sec:two_dim_coboundary_expanders_F2} for $\FF_2$. A potentially easier task is the following:
\begin{problem}[Bounded cocycle expanders with large cosystoles in permutations] \label{prob:6.6.}
    Is there such a family for which
    \[
\forall m\in \mathbb{N}\ \colon \ \ h_0(\cX_m,\Sym),h_1(\cX_m,\Sym),CoSyst_1(\cX_m,\Sym)\ge \lambda,
\]
where $\lambda>0$ is independent of $m$?
\end{problem}
 Solving Problem \ref{prob:6.6.} will be the exact Sym analogue of \cites{kaufman2016isoperimetric} with $\FF_2$ coefficients.

\subsection{Random models}
Both random simplicial complexes and random group presentations received significant attention in the recent two decades (cf. \cites{Ollivier2005AJ2,ollivier2010random,kahle2016random} and the references therein). It is natural to ask what is the expected stability rate of a random object. It turns out that these kinds of enquiries are closely related to Problem \ref{prob:sofic_groups}, namely, finding non-sofic groups. A companion paper to this one \cite{CL_stability_Random_complexes} is devoted to  the study of cocycle expansion of random simplicial complexes in the  Linial--Meshulam model \cite{linial_meshulam2006homological}. In that paper, it is proved that if the expected cocycle Cheeger constant with permutation coefficients of a random simplicial complex in the \emph{mid-range} is bounded away from zero, then there are non-sofic hyperbolic groups. This would resolve both, the \emph{soficity problem} as well as the \emph{residual finiteness of hyperbolic groups} problem (cf. \cite{kapovich2000equivalence}).\footnote{Compare also to Conjecture 2.8 in \cite{arzhantseva2014asymptotic} and the references therein.} This result is somewhat parallel\footnote{The result in \cite{CL_stability_Random_complexes} is weaker than Dogon's in the sense that a better  stability rate of a random complex is assumed there. But, under these stronger assumptions, the existence of non-sofic \textbf{hyperbolic} groups is deduced.} to a recent work of Dogon \cite{dogon2023flexible}, where he proved that homomorphism stability in unitaries with any rate in the mid-range will imply the existance of non-hyperlinear groups, which are in turn non-sofic. Both results call for further study of the stability rate of random objects. We do suspect that a random object should not be stable in the mid-range, similar to the way a random code is not locally testable \cite{ben2003some}.

 \subsection{Efficient Stability of finite group presentations}\label{sec:efficient_stability}

Short presentations of finite groups were the focus of many studies --- see the introduction of \cite{guralnick2008presentations}. 
Though finding short presentations of a given (family of) finite group(s) is a natural problem, this search seems to be in conflict with these presentations having a (uniformly) good stability rate. For example, the work of Ben-Sasson--Guruswami--Kaufman--Sudan--Viderman \cite{ben2010locally}, which proves that locally testable codes need to be \emph{redundant}, can be interpreted --- in the language developed further in this section --- as a proof that a specific type of short presentations of certain finite abelian groups must have a cocycle Cheeger constant that tends to zero with their rank. On the other hand, the most expansive presentation --- the multipilication table one --- has,  by \cite{BC22}, a  lower bound on its Cheeger constant which is independent of the specific group. This suggests a tradeoff between \emph{shortness} of a presentation and its \emph{stability}. It will be interesting to understand this tradeoff better. 

Let us be more concrete. 
Recall that for every given finite group, and any given presentation of it, its cocycle Cheeger constant in dimension $1$ with permutation coefficients is positive. But, as was proved in Claim \ref{claim:cheeger_arbitrarily_small}, one can choose presentations with a
Cheeger constant  arbitrarily  close to $0$.
Consider a family of finite groups $G_i$, given by presentations $\langle S_i|R_i\rangle$. The goal is  to relate the complexity of these presentations and their  stability rate  in a uniform manner. Here, complexity can have several interpretations: It can be taken, for example, to be $\log|S|\cdot \sum_{r\in R}|r|$, where $|r|$ is the length of the relation $r$, or to be the bit length of an encoding of the presentation (which can be shorter due to the fact exponents can be written in binary and not unary). 

For a first example, take the groups $\FF_2^k$ for $k\to \infty$. On the one hand,  their multiplication  table presentations,  which are exponentially more complicated than the shortest possible presentations, have a cocycle Cheeger constant of at least $\frac{1}{3000}$ regardless of $k$. 
On the other hand, the much shorter presentation on $k$ generators, where  every generator is required to be of order $2$, and every pair of generators are required to commute, have 
 cocycle Cheeger constants which are  $O(\frac{1}{k})$ (see \cites{kozlov2019quantitative,CVY_efficient}). Furthermore, in \cite{ben2010locally} it is shown that if $\FF_2^k$ is presented as a quotient of $\FF_2^n$ (equipped with the above short presentation) and one adds only $n-k$ new relations and the associated code has linear distance, then the Cheeger constant  tends to zero with $k\to \infty.$ 

The second natural example to study, which is completely open, is that of finite simple groups. Clearly, studying the stability rate of various standard presentations of the alternating and symmetric groups seem to be a natural problem. But, we want to highlight the following case: The groups $SL_n(\FF_p)$ (and more generally, groups of Lie type) have two outstanding families of presentations ---
\begin{enumerate}
    \item \textbf{The Steinberg presentations}: Let $e_{ij}(a)$ be the matrices with $a\in \FF_p$ in the $ij^{\rm th}$ position, and $0$ evereywhere else.  For every $i\neq j$ and $a\in \FF_p$, let $U_{ij}(a)=\Id +e_{ij}(a)$. Take the $U_{ij}(a)$'s as generators, and the relations to be $U_{ij}(a)U_{ij}(b)=U_{ij}(a+b)$ and for every $i\neq j, k\neq \ell$ and $i\neq \ell$, $[U_{ij}(a),U_{k\ell}(b)]={\bf 1}_{j=k} U_{i\ell}(a\cdot b)$. The complexity of this presentation is $\Omega(p)$, which is polynomial in the size of the group --- namely, an expansive presentation.
    \item \textbf{The Guralnik--Kantor--Kassabov--Lubotzky presentations}: Here, we do not describe the presentations in full, but describe their complexity. There is a constant $C$, such that regardless of $p$ and $n$, the presentations use less than $C$ generators, less than $C$ relations, and the length of the relations is at most $C(\log n+\log p)$. See \cites{guralnick2008presentations,lubotzky2010presentations} for all the details. The complexity of these presentations is logarithmic in the size of the group. Note that when one fixes $p$ and lets $n\to \infty$, these presentations are of complexity which is even $\log\log$ the size of the group.
\end{enumerate}
In the spirit of the more understood abelian case, we conjecture the following:
\begin{conj}
    There is a uniform lower bound on the cocycle Cheeger constant in permutations of the Steinberg presentations, at least when fixing $n$ and tending $p$ to $\infty$. On the contrary,   the cocycle Cheeger constants in permutations of the GKKL presentations tend to zero when either $n$ or $p$ tending to $\infty$.
\end{conj}
\begin{rem}
    Note that the length of the relations in the GKKL presentations is not bounded. This is not natural in the viewpoint of property testing, since it implies a non-bounded query complexity for the tester. Nevertheless, one can study the triangulated versions of  these presentations, and amend this problem.
\end{rem}

We now want to make the tradeoff problem between shortness and stability rate even more concrete. To that end, recall the recent MIP*=RE breakthrough \cite{MIPRE} in quantum information theory, which resolved Connes' embedding problem (cf. \cite{Cap_Lup_Sofic_Hyperlinear_book}). A major technical tool, which was developed and used in this work, is a quantum soundness result on a specific two player non-local game based on the individual degree Reed--Muller code \cite{quantum_soundness_tensor_codes}. As is explained in  \cite{CVY_efficient}, this quantum soundness result can be interpreted as a lower bound on the cocycle stabilty rate of certain presentations of $\FF_2^k$, but with coefficients in unitaries equipped with the Hilbert--Schmidt metric rather than in permutations equipped with the Hamming metric. 
Inspired by this deep result, we suggest the following parameter setting as a  challenge.

\begin{problem}[Efficient stability in permutations]\label{prob:efficient_stability}
    Find a family of finite groups $\{G\}$, with  presentations $G\cong\langle S| R\rangle$, such that:
    \begin{itemize}
        \item \emph{The presentation is (somewhat) short}: 
        \[\begin{split}
            |S|,|R|&=\textrm{polylog}(|G|) \\
            \forall r\in R\ \colon |r|&=\textrm{polylog}(|G|).
        \end{split}
        \]
            \item \emph{The stability rate is good}: $\langle S|R\rangle$ is $\rho$-stable, and $\rho(\eps)=\textrm{polyloglog}(|G|)\cdot \eps^c$ for some $c>0$ independent of $|G|$.
    \end{itemize}
\end{problem}

\begin{rem}\label{rem:on_efficient_stability}
It would be more natural to ask for $\rho(\eps)\leq C\eps^c$ for two universal constants $c,C>0$. But, because of gap amplification techniques such as parallel repetition, the $\textrm{polyloglog}(|G|)$ term is good enough  for applications.
\end{rem}

The motivation for this parameter setting is that it allows for \emph{randomness generation} in certain contexts --- see the way
 \cite{quantum_soundness_tensor_codes} is used in \cite{MIPRE}.  Though Problem \ref{prob:efficient_stability} is of intrinsic interest, we are interested in it because it may help to resolve the Aldous--Lyons conjecture regarding the existence of non-cosofic invariant random subgroups of the free group \cite{Aldous_Lyons_Conj} (see the preprint \cite{BCLV_subgroup_tests}).
We now formulate a more specific version of Problem \ref{prob:efficient_stability}, which would be enough for this  application.

\begin{problem}[Locally testable codes with permutation coefficients]\label{prob:LTC_with_perm_coeff}
    Given a matrix $A\in M_{m\times n}(\FF_2)$, one can define the following presentation of $\FF_2^k$, where $k=n-\textrm{rank}(A)$: The set of generators is $S=\{x_i\}_{i\in [n]}$. The set of relations $R$ will contain:
    \begin{itemize}
        \item \emph{Order $2$ relations}: $\{x_i^2\}_{i\in [n]}$.
        \item \emph{Commutation relations}: $\{x_ix_jx_i^{-1}x_j^{-1}\}_{i,j\in [n]}$.
        \item \emph{Linear relations induced by $A$}: $\left\{\prod x_j^{A_{ij}}\right\}_{i\in [m]}$.
    \end{itemize}
 Now, a code $\Ker(A)$ is $\rho$-locally testable (when tested by the  matrix tester --- Algorithm 2 in Section 3.1 of \cite{CL_part1}) if and only if  it is  $\rho$-stable for the above presentation, when the range of the input functions is restricted to $\Sym(2)$. 
 What about the case where the inputs are not restricted to be with range $\Sym(2)$, but any $\Sym(n)$? Namely, what can be said about the homomorphism stability rate of the above presentation for a given code? 
 For the Hadamard code the resulting presentation is essentially the multiplication table of $\FF_2^k$, and is thus stable but expansive. What about more \emph{reasonably sized} codes, such as Reed--Muller codes?
Or, the recently constructed good locally testable codes \cite{LTC_DELLM,LTC_Panteleev_Kalachev}?
\end{problem}

One of our motivations  to study covering stability was that, in this framework, one can potentially translate stability results with unitary coefficients equipped with the Hilbert--Schmidt metric to the discrete setup of permutations with normalized Hamming metric. To that end, it would be useful to have a Poincare-type inequality (cf. Lemma 5.4 in \cite{CL_part1}) on the metric space $U(n)/\Sym(n)$, with constants independent of $n$. We learened from Assaf Naor that the metric space $M(n)/\Sym(n)$ \textbf{does not} have a Poincare-type inequality on it which is independent of $n$, due to \cite{andoni2018snowflake}. But, the case of $U(n)/\Sym(n)$ seems to be open, and will be sufficient for us to translate the results of \cite{quantum_soundness_tensor_codes} to the permutation setup.

\appendix
\section{\textbf{Weighted Cheeger lower bound}}\label{appendix:weighted_Cheeger}
    The following argument is \textbf{very} standard. It appeared in its uniformly weighted form in Proposition 5.5 and Remark 5.7 in Part I of this paper \cite{CL_part1}, using the $L^2$-Poincare inequality described in Lemma 5.4.  See also \cites{Dodziuk_Cheeger,Alon_Cheeger_2,ALON_Milman_Cheeger_1} for the original proofs. The main idea is to use the Reyleigh quotient definition of the second largest eigenvalue. We add it only for completeness. 
    
\begin{proof} [Proof of Proposition \ref{prop:weighted_Cheegr_lower_bound}]
    Let $\alpha\colon V(\cX)\to \FF_2$ be a non-coboundary $0$-cochain. Let $B$ be the support of $\alpha$, and let $\beta=\mu_0(B)$. Define $f\colon V(\cX)\to \mathbb{R}$ as follows: 
    \[
f(v)=\begin{cases}
    1-\beta & v\in B,\\
    -\beta & v\notin B.
\end{cases}
    \]
    Note that $$\langle f,{\bf 1}\rangle=\Ex_{v\sim \mu_0}[f(v)]=\mu_0(B)(1-\beta)-\mu_0(\bar B)\beta=0,$$ where $\bar B$ is the complement of $B$ in $V(\cX)$, and that 
    \[
    \langle f,f\rangle =\Ex_{v\sim \mu_0}[f(v)^2]=\beta(1-\beta)^2+(1-\beta)\beta^2=\beta(1-\beta).
    \]
    Hence, by the  Rayleigh quotient definition of the second eigenvalue of an operator, we have 
    \[\langle f,Af\rangle \leq \lambda \langle f,f\rangle=\lambda \beta(1-\beta).
    \]
    On the other hand, 
    \[
    \begin{split}
    \langle f,Af\rangle &=\Ex_{xy\sim \mu_1}[f(x)f(y)]\\
    &=-2\mu_1(E(B,\bar B))\beta(1-\beta)+\mu_1(E(B,B))(1-\beta)^2+\mu_1(E(\bar B,\bar B))\beta^2,
    \end{split}
    \]
    where $E(C,D)$ is the set of oriented edges in $\overrightarrow{\cX}(1)$ with origin   in $C$ and endpoint in $D$. Note that by definition, $2\mu_1(E(B,\bar B))=\mu_1(E(B,\bar B))+\mu_1(E(\bar B,B))=\Vert \delta \alpha\Vert$. Furthermore, $d(\alpha,Z^0(\cX,\FF_2))=\min(\beta,1-\beta)$. Now, when sampling an edge $xy\sim \mu_1$, the probability $x\in B$ is  $\beta$ by the fact $\mu_0$ is descending from $\mu_1$. Hence, $\mu_1(E(B,B))=\mu_1(B,\cX(0))-\mu_1(E(B,\bar B))=\beta-0.5\Vert \delta \alpha\Vert$. Similarly, $\mu_1(E(\bar B,\bar B))=(1-\beta)-0.5\Vert \delta \alpha \Vert$. All in all,
   \[
    \begin{split}
    \langle f,Af\rangle &=-\Vert \delta \alpha\Vert\beta(1-\beta)+(\beta-0.5\Vert \delta \alpha\Vert)(1-\beta)^2+(1-\beta-0.5\Vert \delta \alpha\Vert)\beta^2\\
    &=-0.5\Vert \delta \alpha\Vert+\beta(1-\beta).
    \end{split}
    \]
    Combining all, we get
    \[
    0.5\Vert\delta \alpha\Vert \geq (1-\lambda)\beta(1-\beta)\geq (1-\lambda)\cdot\frac{\min(\beta,1-\beta)}{2},
    \]
which finishes the proof by multiplying both sides of the inequality by $2$.
\end{proof}

\bibliographystyle{plain}
\bibliography{Bib}

\end{document}